\pdfoutput=1 
\documentclass[smallextended]{svjour3}       
\smartqed  
\usepackage{graphicx}
\usepackage{amsmath}
\usepackage{amssymb}
\usepackage{amstext}
\usepackage{mathrsfs}
\usepackage{bm}             
\usepackage{graphicx}       
\usepackage{gensymb}
\usepackage{subeqnarray}
\usepackage[usenames]{xcolor}  
\usepackage{rotating}
\usepackage{paralist}   

\usepackage{algorithmic}
\usepackage[linesnumbered,ruled]{algorithm2e}

\newcommand\SSa{Section}

\newcommand{\Forall} {{\forall\,}}

\newcommand\bifont {\bm}

\newcommand\RR {\mathbb{R}}
\newcommand\IR {\RR}
\newcommand\R  {\RR}    

\newcommand\Om {{\Omega}}
\newcommand\oOm {\bar{\Om}}

\newcommand\dom{{\Gamma}}
\newcommand\gom{{\Gamma}}
\newcommand\gomIO{{\Gamma_{\mathrm{IO}}}}
\newcommand\gomW{{\Gamma_{\mathrm{W}}}}

\newcommand\RKV {R_{K,V}}
\newcommand\RKB {R_{K,B}}
\newcommand\RKVi {R^i_{K,V}}
\newcommand\RKBi {R^i_{K,B}}
\newcommand\RDKV {R_{K,V}^*}
\newcommand\RDKB {R_{K,B}^*}
\newcommand\RDKVi {R_{K,V}^{*,i}}
\newcommand\RDKBi {R_{K,B}^{*,i}}

\newcommand\T{^{\mkern-1.5mu\mathsf{T}}}

\newcommand\dd {\mathrm{d}}
\newcommand\D {\mathrm{D}}

\newcommand\dx {{\,\dd x}}

\newcommand\dS {{\,\dd S}}

\def\krr{{\mathscr R}}

\newcommand\up {u_h^+}
\newcommand\z {z}

\newcommand\uh {{{u}_h}}

\newcommand\zh {{{z}_h}}

\newcommand\nn  {\bifont{n}}
\newcommand\vv {\bifont{v}}

\newcommand\zz {\bifont{z}}
\newcommand\vvp {\bifont{\varphi}}

\newcommand\pp {{\texttt{p}}}
\newcommand\ppn  {\bm{\mathrm{p}}_{\nn}}
\newcommand\vtheta {{\bm{\vartheta}}}
\newcommand\tvtheta {\tilde{\bm{\vartheta}}}
\newcommand\tnn {\tilde{\nn}}

\newcommand\uu {\bifont{u}}
\newcommand\uuRP {\bifont{u}_{\mathrm{RP}}}

\newcommand\ww {\bifont{w}}
\newcommand\w {\ww}
\newcommand\wt {\tilde\w}

\newcommand\wwh {{\ww_h}}

\newcommand\uBC {\uu_{\Gamma}} 
\newcommand\wBC {\w_{\mathrm{BC}}} 

\newcommand\Mir{{\bifont{m}_{\Gamma}}}  

\newcommand\BC{{\mathscr B}}

\newcommand\ppp {{\bifont{p}}}

\newcommand\hp {{h,\ppp}}    

\newcommand\VV {{\bifont{V}}}

\newcommand\basis {{{\mathrm B}_{\hp}}}

\newcommand\vp {\varphi}
\newcommand\va {{\pmb{\varphi}}}  
\newcommand\vah {{\pmb{\varphi}_h}}  
\newcommand\vahT {{\pmb{\varphi}_h^{\mkern-1.5mu\mathsf{T}}}}  
\newcommand\vpsi {{\pmb{\psi}}}  

\newcommand\GG {{\pmb{G}}}

\newcommand\ff {{\bifont{f}}}

\newcommand\F {{\bifont{F}}}
\newcommand\A {\mathbb{A}}
\newcommand\HH {\mathbf{H}}
\newcommand\Hbound {{\HH_{\dom}}}
\newcommand\HboundW {{\HH_{\dom_{\mathrm{W}}}}}

\newcommand\HboundWone {{\HH_{\dom_{\mathrm{W}}}^{1}}}
\newcommand\HboundWtwo {{\HH_{\dom_{\mathrm{W}}}^{2}}}

\newcommand\HboundWi {{\HH_{\dom_{\mathrm{W}}}^{i}}}
\newcommand\HboundIO {{\HH_{\dom_{\mathrm{IO}}}}}

\providecommand{\mesh}{\mathcal{T}_h}

\newcommand\Th {{\mathcal{T}}_h}
\newcommand\Thz{{\mathcal{T}}_{h,0}}
\newcommand\Thm {{\mathcal{T}}_{h,m}}
\newcommand\ThmP {{\mathcal{T}}_{h,m+1}}

\def\bbHj{{\bm{H}^{1}_h}}

\def\norm#1#2{\left\|#1\right\|_{#2} }
\def\normS#1#2{\|#1\|_{#2} }

\newcommand\VVh {{\bm{V}_h}}

\newcommand\llbracket {[\![}
\newcommand\rrbracket {]\!]}

\def\jump#1{\llbracket#1\rrbracket}
\def\aver#1{\{\!\!\{#1\}\!\!\}}

\newcommand\bS   {{\bifont{S}}}
\newcommand\bShp {{\bS_{h}^{\pp}}}
\newcommand\bShpp {{\bS_{h}^{\pp+1}}}

\newcommand\bShzp {\bS_{h,0}^\pp}
\newcommand\bShmp {\bS_{h,m}^\pp}
\newcommand\bShmP {\bS_{h,m+1}^\pp}
\newcommand\bShmpp {\bS_{h,m}^{\pp+1}}

\newcommand\Nhp {{N_{\hp}}}

\newcommand\ah {{a_h}}
\newcommand\ahP {{a_h^{\prime}}}
\newcommand\sahL {{a_h^{\scriptscriptstyle\mathrm{L}}}}

\newcommand\bbh {{\bka_h}}
\newcommand\bbhP {{\bka_h'}}
\newcommand\ahL {{\bka_h^{\scriptscriptstyle\mathrm{L}}}}
\newcommand\bbht {{\tilde{\bka}_h}}

\newcommand\matA{\mathbb{A}}

\newcommand\matC{\mathbb{C}}

\newcommand\matG{\mathbb{G}}
\newcommand\matH{\mathbb{H}}

\newcommand\matL{\mathbb{L}}
\newcommand\matM{\mathbb{M}}
\newcommand\matN{\mathbb{N}}

\newcommand\matQ{\mathbb{Q}}
\newcommand\matP{\mathbb{P}}
\newcommand\matR{\mathbb{R}}
\newcommand\matT{\mathbb{T}}
\newcommand\matU{\mathbb{U}}

\def\PPP {\bifont{P}}

\newcommand\W {{\pmb{\xi}}}
\def\bnul{{\bf 0}}

\newcommand\chij {{\zeta_1}}
\newcommand\chid {{\zeta_2}}
\newcommand\chit {{\zeta_3}}
\newcommand\chic {{\zeta_4}}
\newcommand\chiLi {{\zeta_i^{\rm{L}}}}
\newcommand\chijL {{\zeta_1^{\rm{L}}}}
\newcommand\chidL {{\zeta_2^{\rm{L}}}}

\newcommand\chitiL {{\zeta_3^{i,\rm{L}}}}
\newcommand\chiti {{\zeta_3^{i}}}
\newcommand\chicL {{\zeta_4^{\rm{L}}}}

\newcommand\fluxVS {{\HH_{\mathrm{VS}}}}

\newcommand\HboundLOne {{\HH_{\dom_{\mathrm{W}}}^{\mathrm{1,L}}}}
\newcommand\HboundLTwo {{\HH_{\dom_{\mathrm{W}}}^{\mathrm{2,L}}}}

\newcommand\HboundLI {{\HH_{\dom_{\mathrm{W}}}^{\mathrm{i,L}}}}

\def\pd{\partial}

\newcommand\press {\mathrm{p}}

\newcommand\wh {\w_h}

\newcommand\whb {\bar{\w}_h}
\newcommand\zzh {{\zz_h}}

\newcommand\whm {\w_{h,m}}
\newcommand\zzhm {\zz_{h,m}}

\newcommand\whmp {\w_{h,m}^+}
\newcommand\zzhmp {\zz_{h,m}^+}

\newcommand\whp {\w_h^+} 
\newcommand\zhp {\zz_h^+}
\newcommand\wKp {\w_K^+} 
\newcommand\zKp {\zz_K^+}

\newcommand\whN {\w_h^{0}} 
\newcommand\whk {\w_h^{k}} 
\newcommand\whkP {\w_h^{k+1}} 
\newcommand\dhk {\bifont{d}_h^k}

\newcommand\defD {{\mathcal D}}

\newcommand\dK {\partial K}
\newcommand\dKI {\partial K \backslash \gom}

\newcommand{\vertical}[1]{\begin{sideways}{#1}\end{sideways}}

\newcommand{\DoF} {{\mathrm{DoF}}}

\newcommand\Br {\mathrm{B}_1}

\def\bka{\bifont{a}} 
\def\bkn{\bifont{n}}

\newcommand\J{J}
\newcommand\Jd{J_h}
\newcommand\JdL{J_h^{\mathrm{L}}}
\newcommand\JP {{J^{\prime}}}

\newcommand\plus{^{\scriptscriptstyle (+)}}
\newcommand\minus{^{\scriptscriptstyle (-)}}

\newcommand\Plus{^{+}}
\newcommand\Minus{^{-}}

\providecommand{\intx}[4]{\int_{#1}^{#2}#3\,\mathrm{d}#4}

\providecommand{\intS}[2]{\int_{#1}#2 \,\dS}

\newcommand\Vp {\bifont{\tilde V}}
\newcommand\Vh {V_h}

\newcommand\WW {{\bifont{H}}}
\newcommand\WWh {{\WW}}

\newcommand\Rtri {\mathcal{R}_h^{(3)} }

\providecommand{\zzhi}[1]{z_h^{#1} }

\providecommand{\HboundMatrix}[1]{\matH_{\mathrm{W}}^{\mathrm{#1,L}}}

\providecommand{\resDFK}[2]{\bifont{R}^*_K(#1, #2) }
\providecommand{\resDFI}[2]{\bifont{r}^*_K(#1, #2) }
\providecommand{\resFK}[1]{\bifont{R}_K(#1) }
\providecommand{\resFE}[1]{\bifont{r}_K(#1) }

\providecommand{\resDFKi}[2]{\bifont{R}^{*,i}_K(#1, #2) }
\providecommand{\resDFIi}[2]{\bifont{r}^{*,i}_K(#1, #2) }
\providecommand{\resFKi}[1]{\bifont{R}^i_K(#1) }
\providecommand{\resFEi}[1]{\bifont{r}^i_K(#1) }

\providecommand{\resEul}[1]{r_h(\wwh)(#1) }
\providecommand{\resEulD}[1]{r_h^*(\wwh, \zzh)(#1) }
\providecommand{\resEulTrue}[1]{r_h(\ww)(#1) }
\providecommand{\resEulDTrue}[1]{r_h^*(\ww, \zz)(#1)  }

\newcommand\etaI {\eta^{\rm I}}
\newcommand\etaIK {\eta^{\rm I}_K}
\newcommand\etaII {\eta^{\rm I\!I}}
\newcommand\etaIIK {\eta^{\rm I\!I}_K}

\newcommand\etaIIIK {\eta^{\rm I\!I\!I}_K}

\newcommand\TOL {\mathrm{TOL}}

\newcommand\U {{\pmb{\xi}}} 

\newcommand\diag{\mathrm{diag}}

\def\dim{{\rm dim}}

\newcommand\restr[2]{{
  \left.\kern-\nulldelimiterspace 
  #1 
  \vphantom{\big|} 
  \right|_{#2} 
  }}
  
\newcommand\tleft[2]{{
  \left.\kern-\nulldelimiterspace 
  #1 
  \vphantom{\big|} 
  \right|^{#2}_{-} 
  }}
  
\newcommand\tright[2]{{
  \left.\kern-\nulldelimiterspace 
  #1 
  \vphantom{\big|} 
  \right|^{#2}_{+} 
  }}

\newcommand{\cD}{{c_{\mathrm{D}}}}
\newcommand{\cL}{{c_{\mathrm{L}}}}
\newcommand{\cM}{{c_{\mathrm{M}}}}
\newcommand{\Min}{{M_{\infty}}}
\newcommand {\wrt} {{w.\,r.\,t.\ }}

\begin{document}

\title{Goal-oriented anisotropic $hp$-adaptive discontinuous Galerkin method for the Euler
  equations\thanks{This work was supported by grant No. 20-01074S of the Czech Science Foundation.}
}

\titlerunning{Goal-oriented anisotropic $hp$-adaptive DGM for the Euler equations}        

\author{V{\'\i}t Dolej{\v s}{\'\i}  \and
        Filip Roskovec 
}


\institute{V. Dolej{\v s}{\'\i} \at
  Charles University,  Faculty of Mathematics and Physics,
  Sokolovsk\'a 83, 
  186 75 Prague, 
  Czech Republic,
  \email{dolejsi@karlin.mff.cuni.cz}           
  \and
  F. Roskovec \at
  Charles University, Faculty of Mathematics and Physics,
  Sokolovsk\'a 83, 
  186 75 Prague, 
  Czech Republic 
  \email{roskovec@gmail.com}
}

\date{Received: date / Accepted: date}

\maketitle

\begin{abstract}
  We deal with the numerical solution of the compressible Euler equations
  with the aid of the discontinuous Galerkin (DG) method with
  focus on the goal-oriented error estimates and adaptivity.
  We analyze the adjoint consistency of the DG scheme where the dual problem
  is not formulated by the differentiation of the DG form and the target
  functional but using a suitable linearization of the nonlinear forms.
  Further, we present the goal-oriented anisotropic $hp$-mesh adaptation technique
  for the Euler equations. The theoretical results are supported by numerical
  experiments. 
  \keywords{Euler equations \and discontinuous Galerkin method \and target functional \and
    adjoint consistency \and anisotropic $hp$-mesh adaptation}
 \subclass{65N30 \and 65N50 \and 76H05}
\end{abstract}


\section{Introduction}
\label{sec:intro}

The motion of inviscid compressible fluids is described by the Euler equations whose
efficient numerical solution is still a challenging problem. Among the prominent methods
belongs the discontinuous Galerkin method (DGM) which exhibits an efficient technique for the
numerical solution of various partial differential equations (PDEs).
This paper looks into the solution of the Euler equations by DGM with 
focus on the goal-oriented (or output-based) error estimators and the mesh adaptivity.
The aim is the approximation of a quantity of interest (e.g., the aerodynamics coefficients)
given by a target functional depending on the solution of the problem considered.

The framework of the goal-oriented error estimates is already very-well established
approach namely for linear PDEs, cf. surveys in
\cite{RannacherBook,BeckerRannacher01,GileSuli02}.
This technique is based on the formulation of an adjoint problem whose solution
appears as a weight of the primal residuals.
For a nonlinear PDE (and a nonlinear target functional), the corresponding
adjoint problem has to be derived for a linearized primal problem.
The linearization is carried out by
a differentiation of the primal formulation with respect to its approximate solution,
see \cite{RannacherBook,BeckerRannacher01,GileSuli02} for a general framework and, e.g.,
\cite{Hartmann2006Derivation,HH06:SIPG2,LoseilleDervieuxAlauzet_JCP10,ceze:2013aniso,BalanWoopenMay16} for applications of this approach to compressible flow problems.
A comprehensive review of goal-oriented error estimates and mesh adaptation
in computational fluid dynamics
is given in \cite{FidkowskiDarmofal_AIAA11}.

Let $\ah$ be a semi-linear form representing the discretization of the given (nonliner) PDE,
then the approximate solution $\uh \in \Vh$ is defined by
\begin{align} 
\label{mod:primalh} 
 \ah(\uh, \vp_h) = 0 \qquad \forall \vp_h \in \Vh,
\end{align}
where $\Vh$ is a finite element space.
Further, let $J$ be the target functional defining the quantity of interest $J(u)$ where $u$
is the exact solution. In order to estimate the error $J(u) - J(\uh)$, we introduce
the adjoint (or dual) problem
\begin{align} 
\label{mod:dualh}
\ahP[\uh](\psi_h,\zh) = \JP[\uh](\psi_h) \qquad \forall \psi_h \in \Vh,
\end{align}  
where $\ahP[\uh](\cdot,\cdot)$ and $\JP[\uh](\cdot)$ denote the Fr{\'e}chet derivatives
of $\ah$ and $J$ at $\uh$, respectively, and $\zh$ is the approximate solution of the adjoint
problem.
The form $\ahP[\uh](\cdot,\cdot)$ represents the Jacobian corresponding to
the nonlinear algebraic system \eqref{mod:primalh} and it is also employed for the iterative
solution of \eqref{mod:primalh} by the Newton method.
However, the evaluation of $\ahP$ requires a differentiation
of $\ah$ which is a little laborious and moreover a differentiable numerical flux
is required. Otherwise, an approximation, e.g., by finite differences has to be employed,
cf. \cite{Hartmann2005Role}.

The key properties in goal-oriented error estimates are the consistency and
the adjoint consistency of the numerical scheme which
means that identities \eqref{mod:primalh} and \eqref{mod:dualh} are valid also
for the exact solutions $u$ and $z$ of the primal and adjoint problems, respectively.
It was demonstrated in \cite{Harriman2004importance} that the lack of the adjoint consistency
leads to the decrease of the rate of convergence.
In \cite{Hartmann2007Adjoint} the detailed analysis of the adjoint consistency
of the DGM
for the compressible Euler and the Navier-Stokes equations was presented. It was shown that
the realization of the boundary conditions have to be carefully chosen and,
in some cases, the target functional $J$ has to be adopted to the numerical discretization.

In \cite{st_estims_NS,stdgm_est}, we developed an implicit discretization
of the Euler and Navier-Stokes equations by DGM where the arising nonlinear system
is solved by an iterative method.
Instead of the Jacobian $\ahP$, we employed a suitable consistent linearization $\sahL$
of the form $\ah$.
The algebraic representation of $\sahL$ is called the flux matrix,
its sparsity is the same as the Jacobian and 
this approach does not require the differentiability of the numerical fluxes.
Therefore, there is a natural question if it is possible
to replace $\ahP$ in \eqref{mod:dualh} by $\sahL$ too.
Particularly, if the use of $\sahL$ in \eqref{mod:dualh} still preserves the adjoint
consistency of the numerical scheme. The positive answer would justify the use
of other iterative schemes (except the Newton method) in the goal-oriented
computations.

In this paper we analyze the aforementioned method
using the linearized form $\sahL$ instead of Jacobian
for the compressible Euler equations.
The core of the adjoint consistency is the appropriate treatment of boundary conditions
and the setting of the adjoint problem. We consider here two ways of the realization of
the impermeable boundary condition, the first one is the same as in \cite{Hartmann2007Adjoint}
and the second one is based on the mirror operator as in, e.g., \cite{bas-reb-JCP,harthou2}.
Each of these treatments requires a different modification of the target functional.

The next novelty is this paper is the goal-oriented anisotropic $hp$-mesh adaptation
technique for the Euler equations.
The $h$-variant of this approach (usually for $p=1$ only) was treated in many papers, e.g.,
\cite{VendittiDarmofal_JCP02,VendittiDarmofal_JCP03,LoseilleDervieuxAlauzet_JCP10,Fidkowski11,YanoDarmofal12}. 
In \cite{DWR_AMA,ESCO-18}, we employed some ideas of these papers and
developed the $hp$-variant for linear problems.
Here, we present its extension to the Euler equations which
is relatively straightforward. A different type of anisotropic $hp$-adaptation method
was published recently in \cite{RanMayDol-JCP20} following some ideas from \cite{BalanWoopenMay16}.
Let us also mention papers \cite{HartmanLeicht,ceze:2013aniso} dealing with anisotropic mesh
adaptation for structured (quadrilateral) grids.

The contents of the rest of the paper is the following. In Section~\ref{sec:model},
we briefly recall the Euler equations with several important properties.
In Section~\ref{sec:dual}, we define the quantities of interest
(drag, lift and momentum coefficients) and the corresponding continuous adjoint problems.
Further, Section~\ref{sec:eulDG} contains the discontinuous Galerkin discretization
of the Euler equations with focus on the treatment of boundary conditions.
Section~\ref{sec:EulerNewton} contains the iterative solver for the primal problem
and the definition of the corresponding adjoint problem. Moreover, we prove here
the adjoint consistency of the method which is the main theoretical results of this paper.
Furthermore, in Section~\ref{sec:EE}, we present the resulting goal-oriented error estimates
and the anisotropic $hp$-mesh adaptation algorithm. Finally, numerical experiments
supporting the adjoint consistency and the performance of the mesh adaptive algorithm are
presented in Section~\ref{sec:numer}.
The summary of the results is given in Section~\ref{sec:concl}.
For simplicity, in the whole paper, we restrict ourselves to the case $d=2$ 
but the the majority of actions below can be easily generalized also to $d=3$.

\section{Inviscid compressible flow model}
\label{sec:model}

We recall the Euler equations describing the steady-state flow of
inviscid compressible fluids together with several useful relations which are employed
in the definition of the numerical scheme and the setting of the adjoint problem.

\subsection{Euler equations}
\label{sec:euler}

Let $\Om \subset \matR^2$ be a bounded domain with its boundary $\dom$
occupied by an inviscid compressible fluid. The steady-state flow 
is described by the Euler equation written as
\begin{align}
  \label{eul:EE}
  \sum_{s=1}^{2}\frac{\partial\ff_s(\w)}{\partial x_s}  = 0,
\end{align}
where $ \w=(w_1,w_2,w_3,w_4)\T = (\rho,\rho v_1,\rho v_2,E)\T$
is the {\em state vector} 
and 
\begin{align}
  \label{eul:fs}
  \ff_1(\w) &= \left(\rho v_1,\ \rho v_1^2+\press,\ \rho v_1 v_2,\ (E+\press) v_1 \right)\T, \\
  \ff_2(\w) &= \left(\rho v_2,\ \rho v_1 v_2,\ \rho v_2^2+\press,\ (E+\press) v_2 \right)\T\notag
\end{align}
are the {\em inviscid fluxes}.
We use the standard notation:
$\rho$\,-\,density, $\press$\,-\,pressure 
(symbol $p$ denotes the degree of polynomial approximation),
$E$\,-\,total energy, 
$\vv=(v_1, v_2)$\,-\,velocity vector with its components and
the superscript $\T$ denotes the transpose of a matrix or a vector.

To the above system, we add the thermodynamical relations defining the pressure 
$ \press=(\gamma-1) (E-\rho\vert\vv\vert^{2}/2), $
and the total energy  $  E=\rho(c_v\theta+\vert\vv\vert^{2}/2)$,
where $\theta$\,-\,absolute temperature, $c_v >0$\,-\,specific heat at constant volume,
and  $\gamma > 1$\,-\,Poisson adiabatic constant.
Finally, the {\em speed of sound} $a$ and the {\em Mach number} $M$ are given by
$a= \sqrt{\gamma \press/\rho}$ and $M=|\vv|/{a}$, respectively.

Furthermore, we define the space of physically admissible state-vectors 
$\w=(w_1,\ldots,w_4) \T$ such that their density and pressure are positive, i.e., 
\begin{align}\label{defD}
   \defD:=\Big\{\w\in\matR^{4};\ w_1=\rho>0,\
   w_4-\tfrac{1}{2w_1} (w_2^2+w_3^2) =\press/(\gamma-1)>0\Big\}.
\end{align}
Obviously, $\ff_s\in [C^{1}(\defD)]^{4}$, $s=1,2$.
Moreover, we introduce the spaces of vector-valed functions
\begin{align}
  \label{spaces0}
  \WWh &:= [H^1(\Om)]^4,\qquad \VV:=\{\w\in \WWh;\ \w(x)\in\defD\ \mbox{a.e. } x\in\Om\},
\end{align}
where $H^1(\Om)$ is the usual Sobolev space of functions having square integrable first weak
derivatives.

The equations \eqref{eul:EE} are accompanied by suitable boundary conditions
written formally in the form
$  \BC(\w) = 0$ on $\gom$.
They are discussed in \SSa~\ref{sec:BC}.

\subsection{Useful relations}

We introduce several properties of the inviscid fluxes  \eqref{eul:fs}.
For details, we refer, e.g., to \cite[Chapter~8]{DGM-book}.
Let  $\A_s(\w):=\frac{\D\ff_s(\w)}{\D\w}$ be the $4\times 4$ Jacobi matrix 
of the mapping $\ff_s$, $s=1,2$. Then  
\begin{align}
  \label{eul:A_s}
  \ff_s(\w) = \A_s(\w) \w,\quad s=1,2,\ \w\in\defD.
\end{align}

Furthermore, let $\nn\in\Br:=\{\nn \in \IR^2;\,|\nn| = 1\}$ be a unit vector, then we define 
the flux of the state vector $\w$ in the direction $\nn$ by 
\begin{align}\label{eul:Pf}
  \PPP(\w,\nn) &  := \sum_{s=1}^{2} \ff_s(\w) n_s 
   =  \left(
  \begin{array}{c}
    \rho \vv \cdot\nn \\
    \rho v_1 \vv \cdot\nn+ \press n_1 \\
    \rho v_2 \vv \cdot\nn+ \press n_2 \\
    (E+\press) \vv \cdot\nn  \\
  \end{array}
    \right). 
\end{align}
Obviously, the Jacobi matrix $\frac{\D\PPP(\w,\nn)}{\D\w}$ can be expressed in the form
\begin{align}\label{eul:PP}
\matP(\w,\nn) := \frac{\D\PPP(\w,\nn)}{ \D\w} = \sum_{s=1}^{2} \frac{\D\ff_s(\w)}{\D\w} n_s = 
    \sum_{s=1}^{2} \A_s(\w) n_s.
\end{align}


Similarly, as in \eqref{eul:A_s}, we have
\begin{align}
  \label{PPP}
  \PPP(\w,\nn)=\matP(\w,\nn) \w\qquad \forall \w\in\defD\ \forall\nn\in\Br.
\end{align}
Further, matrix $\matP$ is diagonalizable,
hence $\matP = \matT^{-1} \matL \matT$, $\matL= \diag(\lambda_1,\dots, \lambda_4)$.
We  introduce its positive and negative parts $\matP^+$ and $\matP^-$, respectively, by
\begin{align}
  \label{Ppm}
  \matP^{\pm} = \matT^{-1} \matL^{\pm} \matT,\ \matL= \diag(\lambda_1^{\pm},\dots, \lambda_4^{\pm}),
  \ \lambda^+ = \max(\lambda, 0),\, \lambda^- = \min(\lambda, 0).
\end{align}
Obviously, $\matP = \matP\Plus+\matP\Minus$.

\subsection{Boundary conditions} 
\label{sec:BC}

In order to specify the boundary conditions we decompose the boundary $\gom$
into two disjoint parts: the impermeable walls $\gomW$ and {inlet/outlet} $\gomIO$ such that
$\overline{\gom}= \overline{\gomW} \cup \overline{\gomIO}$. In the rest of this paper,
the symbol $\nn=(n_1,n_2)$ denotes the unit outer normal either to the domain boundary $\gom$
or to the boundary of a mesh element. 

On $\gomW$,  we prescribe the impermeability condition $\vv \cdot \nn = 0$,
hence we put
\begin{align}
  \label{eul:BCw}
  \BC(\w) := w_2 n_1 + w_3 n_2 = 0 \qquad \mbox{ on } \gomW.
\end{align}

On $\gomIO$, the number of prescribed boundary conditions depends on the
flow regime (subsonic/supersonic inlet/outlet) and it is equal to the number of incoming
characteristics of the linearized problem, e.g., cf. \cite{FEI2}.
We define the boundary operator $\BC$
and prescribe boundary conditions by
\begin{align}
 \label{eul:BCio} 
 \BC(\w) := \matP\Minus(\w,\nn)(\w - \wBC)  \qquad \mbox{ on } \gomIO,
\end{align}
where $\matP\Minus$ is given by \eqref{Ppm} and $\wBC$ is the prescribed state vector,
e.g., from the far-field boundary conditions.

\section{Quantity of interest and continuous adjoint problem}
\label{sec:dual}
The most interesting target quantities in inviscid compressible flows 
are the drag ($\cD$), lift ($\cL$) and momentum ($\cM$) coefficients. 
In this section we introduce the target functional $\J$ representing either of these coefficients 
in a unified way.

\subsection{Quantity of interest} 

Similarly as in \cite{Hartmann2007Adjoint}, we define the target functional 
by the integral identities
\begin{align} 
   \label{eul:cDL}
  J(\w)  = 
  \intS{\gomW}{\press \, \nn \cdot \vtheta } = \intS{\gomW}{\ppn \cdot \tvtheta } 
  = \intS{\gomW}{ j(\w) },
\end{align}
where $\press$ is the pressure,
$\nn=(n_1,n_2)$ is unit outer normal to $\gomW$,
\begin{align}
  \label{ppn}
  \ppn =\press\, (0,n_1, n_2,0)\T,
\end{align}
$\vtheta = (\vartheta_1,\vartheta_2)\T$ is a given vector in $\R^2$ (cf. hereafter),
$\tvtheta = (0,\vartheta_1,\vartheta_2, 0)\T$ is its extension to $\R^4$
and consequently $j(\w) := \ppn \cdot \tvtheta$. 

The vector $\vtheta$ is given for the drag and lift coefficients by 
\begin{align}
\label{eul:thetaDL}
\vtheta_\mathrm{D} = \frac{1}{C_\infty} (\cos(\alpha), \sin(\alpha))\T \quad \text{ and } \quad
\vtheta_\mathrm{L} = \frac{1}{C_\infty} (-\sin(\alpha), \cos(\alpha))\T, 
\end{align}
respectively, where 
$\alpha$ denotes the angle of attack of the flow,
$C_\infty\! =\! \frac{1}{2} \rho_\infty |\vv_\infty|^2 L$,
$ \rho_\infty$ and $\vv_\infty$ are the far-field density and velocity, 
respectively and $L$ is the reference length.
Therefore, we can write
\begin{align} 
   \label{eul:cDL1}
   \cD=  \intS{\gomW}{\ppn \cdot \tvtheta_\mathrm{D} }, \qquad
   \cL=  \intS{\gomW}{\ppn \cdot \tvtheta_\mathrm{L} }.
\end{align}

The coefficient of momentum is defined as 
\begin{align}
  \cM= \frac{1}{C_\infty L}
  \intS{\gomW}{ \press (x - x_{\rm ref})\times ( \matQ(\alpha) \nn )},
\end{align}
where $x_{\rm ref} \in \Om$ is the moment reference point,
$\matQ(\alpha)$ 
is the rotation matrix trough the angle $\alpha$ in the counterclockwise direction and
$x\times y = x_1 y_2 - x_2 y_1$,  $x,y \in \matR^2$.
In order to obtain the form of the target functional \eqref{eul:cDL}, 
we use the relation
$(x - x_{\rm ref})\times ( \matQ(\alpha) \nn )  = (x - x_{\rm ref}) \matG ( \matQ(\alpha) \nn )$, 
where  $\matG = \left( (0,1)\T, (-1,0)\T \right)$. 
Hence, for the momentum coefficient, we employ in \eqref{eul:cDL}
\begin{align}
 \label{eul:thetaM}
 \vtheta_\mathrm{M}:= \frac{1}{C_\infty L_{\rm ref}}((x - x_{\rm ref}) \matG  \matQ(\alpha) )\T
 \quad\Rightarrow\quad \cM=  \intS{\gomW}{\ppn \cdot \tvtheta_\mathrm{M} }.
\end{align}

Further we define the Fréchet directional derivative
of the target functional $J$ at $\ww\in\VV$ in the direction $\vvp\in\WWh$ given by
\begin{align}
  \label{J:der}
  \JP[\ww](\vvp) = \intS{\gomW}{ j'(\ww)\cdot \vvp },
\end{align}
where $j'(\ww)\in\WWh$ and $j'(\ww)\cdot \vvp = \lim_{t\to 0}\tfrac{1}{t}( j(\ww+t\vvp) - j(\ww))$,
$\vvp\in \WWh$.

The differentiation of pressure $\press=(\gamma-1) (E-\rho\vert\vv\vert^{2}/2)$
with respect to $\w$ gives
\begin{align} 
\label{eul:derP}
\frac{D \press}{D\ww} =
  (\gamma - 1)\begin{pmatrix} 
    \frac{1}{2} |\vv|^2 ,
    -v_1 ,
    -v_2 , 
    1
  \end{pmatrix} \T
\end{align}
and hence the Jacobi matrix of the vector $\ppn$ (cf. \eqref{ppn}) equals
\begin{align*}
  \matP_W(\w,\nn) :=
  \frac{D \ppn}{D\ww} =
  \tnn\otimes \frac{D \press}{D\ww} =
  (\gamma-1)
  \left(\begin{array}{cccc}
      0 & 0 & 0 & 0\\
      |\vv|^2\,n_1/2&\ -v_1n_1 & \ -v_2 n_1 &\ n_1 \\
      |\vv|^2\,n_2/2&\ -v_1n_2 &\ -v_2 n_2 &\ n_2 \\
      0 & 0 & 0 & 0\\
    \end{array}
    \right),
\end{align*}
where 
$\tnn = (0,n_1,n_2,0)\T$ is the extension of $\nn=(n_1,n_2)\T$
and $\otimes$ denotes the vector outer product.
Altogether we get $j'(\ww) = \matP_W(\ww,\nn)\T \tvtheta $ and 
\begin{align} 
\label{eul:DJ}
\JP[\ww](\vvp) = \intS{\gomW}{\tvtheta \T\, \matP_W(\ww,\nn) \vvp }
,\quad \ww\in\VV,\ \vvp\in\WWh.
\end{align}
Let us present several relations between the introduced terms.
\begin{lemma}\label{eul:lemmaP}
 Let $\w \in \defD$, $\vvp\in\R^4$, $\nn=(n_1,n_2)\T$ and $\tnn = (0,n_1,n_2,0)\T$, then
 \begin{subequations}
   \label{lem1}
   \begin{align}
     \label{lem1.1}
     & \matP\T_W(\ww,\nn) \vvp = \frac{D \press(\ww)}{D \ww} \left( \tnn \cdot \vvp \right),\\
     \label{lem1.2}
     &\ppn = \matP_W(\ww,\nn) \ww, \quad \mbox{ where} \quad \ppn = \ppn(\ww), \\
     \label{lem1.3}
     & J(\ww) = \JP[\ww](\ww).
   \end{align}
 \end{subequations}
\end{lemma}
\begin{proof}
  The statement \eqref{lem1.1}
  follows directly from the definition of matrix multiplication if we realize that $\matP\T_W(\ww,\nn) = \frac{D \press(\ww)}{D \ww} \otimes (0,n_1,n_2, 0)\T $. 
  \par
  In order to prove  \eqref{lem1.2}, we can write
  \begin{align*}
    \matP_W(\ww,\nn) \ww = (\gamma - 1 ) r(\ww) (0,n_1,n_2, 0)\T,
  \end{align*}
  where 
 \begin{align*}
  r(\ww) &= \tfrac12 | \vv |^2 w_1 - v_1 w_2 - v_2 w_3 + w_4 \\ \nonumber 
    &= \tfrac12 |\vv |^2 \rho - v_1^2 \rho - v_2^2\rho + E 
    = E - \tfrac12\rho|\vv |^2
    ={\press(\ww)}/{ (\gamma - 1)}.    
 \end{align*}
 The statement \eqref{lem1.3} follows from the definition \eqref{eul:cDL}
 of the target functional $J$, 
 its derivative \eqref{eul:DJ} and  \eqref{lem1.2}.
 \qed
\end{proof}

\begin{lemma} 
\label{eul:lemmaP2}
 Let $\wt, \w \in \VV$ satisfy the impermeability condition \eqref{eul:BCw}, 
 then
 \begin{subequations}
   \label{eq:PPw}
   \begin{align}
     \label{eq:PPw1}
     & \matP(\wt,\nn) \w  = \matP_W(\wt,\nn) \w \quad \text{ on } \gomW, \\
     & \label{eq:PPw2}
     \PPP(\w,\nn) \cdot \tvtheta = \press \nn\cdot \vtheta\quad \text{ on } \gomW,
   \end{align}
 \end{subequations}
 where $\vtheta = (\vartheta_1,\vartheta_2)\T$, 
 $\tvtheta = (0,\vartheta_1,\vartheta_2, 0)\T$, $\vartheta_1,\vartheta_2\in\R$.
\end{lemma}

\begin{proof}
  Equality \eqref{eq:PPw1} can be derived by a direct computation,
  cf. \cite[Section~8.3.1.1]{DGM-book}.
  Furthermore, \eqref{eq:PPw2} follows from identities
  \eqref{PPP}, \eqref{eq:PPw1}, \eqref{lem1.2} and \eqref{ppn} as
  \begin{align*}
    \PPP(\w,\nn) \cdot \tvtheta
    = \matP(\w,\nn)\w\cdot\tvtheta
    = \matP_W(\w,\nn)\w\cdot\tvtheta
    = \ppn \cdot\tvtheta = \press \nn\cdot \vtheta. \hspace{10mm} \qed
  \end{align*}
\end{proof}

\subsection{Formulation of the continuous adjoint problem} 

The continuous adjoint problem to \eqref{eul:EE} was derived, e.g., in \cite{Hartmann2007Adjoint}
but without the treatment of the inlet/outlet boundary conditions. 
For completeness, we briefly derive the weak as well as strong variant
of the adjoint problem following
the approach from \cite{Giles1997Adjoint} and \cite{Hartmann2007Adjoint}.
We multiply \eqref{eul:EE} by $\zz \in \WWh$, 
integrate by parts and finally differentiate at $\ww$ which leads to 
\begin{align}
  \label{eq:DW1}
  \intS{\dom}{ \sum_{s=1}^2 n_s \ff'_s [\ww](\vvp) \cdot \zz }  
 & - \intx{\Om}{}{ \sum_{s=1}^2 \ff'_s [\ww](\vvp) \cdot  \frac{\partial \zz}{ \partial x_s} }{x} 
 = 0 \quad  
 \forall \vvp \in \Vp. 
\end{align}
Here $\ff'_s [\ww](\vvp)$ denotes the Fr{\' e}chet derivative  of $\ff_s$ at $\ww\in\defD$
along the direction $\vvp\in\Vp$, where $\Vp$ is
the subspace of $\VV$ (cf. \eqref{spaces0}) whose functions satisfy
the boundary conditions \eqref{eul:BCw}--\eqref{eul:BCio}, 
see \cite{Lu2005posteriori} for details. Namely we put
 \begin{align} 
\label{eul:Vp}
\Vp := \{ \vvp \in \VV: \, \vp_2 n_1+\vp_3 n_2  = 0 \text{ on } \gomW,\ 
  \matP\Minus(\w,\nn)\vvp = 0 \text{ on } \gomIO\}.
\end{align} 


The Fr{\' e}chet derivative of $\ff_s$ satisfies 
$\ff'_s [\ww](\vvp) = \frac{\D\ff_s(\ww)}{\D\ww} \vvp =\matA_s(\ww) \vvp$, $s=1,2$.
Due to \eqref{eul:A_s} and \eqref{eul:PP} we have
\begin{align}
\intS{\dom}{ \sum_{s=1}^2 n_s \ff'_s [\ww](\vvp) \cdot \zz }  = \intS{\dom}{ \zz\T \matP(\ww,\nn) \vvp  }
\end{align}
and hence, in virtue of \eqref{mod:dualh} and \eqref{eq:DW1}, we define the variational
formulation of the continuous adjoint problem.
\begin{definition}
 Function $\zz \in \WWh$ is the weak solution of the adjoint problem to \eqref{eul:EE} 
  if it satisfies 
\begin{align} \label{eul:weakContDual}
   \intS{\dom}{ \zz \T \matP(\ww,\nn) \vvp    } 
 - \intx{\Om}{}{ \sum_{s=1}^2   \frac{\partial \zz\T}{ \partial x_s}\matA_s(\ww) \vvp  }{x} 
 = \JP[\ww](\vvp)  \quad \forall \vvp \in \Vp. 
\end{align}
\end{definition}

Further, we rearrange this expression individually on $\gomW$ and $\gomIO$.
Since $\vvp\in\Vp$, we have $\matP\Minus(\ww,\nn) \vvp = 0$ on $\gomIO$.
Then due to \eqref{eul:PP} it holds
\begin{align}
  \label{bc:A}
  \intS{\gomIO}{ \zz\T \matP(\ww,\nn) \vvp } 
  = \intS{\gomIO}{ \zz\T \matP\Plus(\ww,\nn) \vvp}.
\end{align}
Similarly, $\vvp\in\Vp$ implies that 
$n_1 \vp_2 + n_2 \vp_3 = 0$ on $\gomW$.
Functions $\ww$ and $\vvp$ satisfy the assumption of Lemma \ref{eul:lemmaP2}, 
hence we may use \eqref{eq:PPw1} which leads to
\begin{align} 
  \label{bc:B}
 \intS{\Gamma_{W}}{ \zz\T \matP(\ww,\nn) \vvp  } = \intS{\Gamma_{W}}{ \zz\T \matP_W(\ww,\nn) \vvp  }.
\end{align}
Inserting \eqref{bc:A} and \eqref{bc:B} in \eqref{eul:weakContDual} and 
employing \eqref{eul:DJ},
we have
\begin{align} \label{eul:weak2}
  \intS{\gomIO}{ \zz\T \matP\Plus(\ww,\nn) \vvp}
  & +\intS{\Gamma_{W}}{ \zz\T \matP_W(\ww,\nn) \vvp  }
 - \intx{\Om}{}{ \sum_{s=1}^2   \frac{\partial \zz\T}{ \partial x_s}\matA_s(\ww) \vvp  }{x} \notag \\
  & = \intS{\gomW}{\tvtheta \T\, \matP_W(\ww,\nn) \vvp }  \qquad \forall \vvp \in \Vp. 
\end{align}
Using \eqref{lem1.1}, we arrange the integrands of integrals over $\gomW$ in
\eqref{eul:weak2} by
$$\matP\T_W(\ww,\nn) \zz = \frac{D \press(\ww)}{D \ww} (0,n_1, n_2,0) \cdot \zz, \qquad
 \matP\T_W(\ww,\nn)\tvtheta = \frac{D \press(\ww)}{D \ww} \nn \cdot \vtheta.$$ 
Then the equality of the both integrals over $\gomW$  leads to  a boundary condition 
\begin{align} 
 n_1 \z_2 + n_2 \z_3 = \nn \cdot \vtheta \qquad \text{on } \gomW.
\end{align}

Therefore, we can write the equation \eqref{eul:weakContDual} in the strong form. 
\begin{definition}
  \label{def:z}
  We say that function $\zz\in C^1(\oOm)$ is the solution of the adjoint problem
  to \eqref{eul:EE} for the given $\ww\in\VV$ if it satisfies
 \begin{align} 
 \label{eul:contDualStrong}
- \sum_{s=1}^2 (\matA_s(\ww))\T \frac{\partial \zz}{ \partial x_s} = 0 &\text{ in } \Om, 
\end{align} 
with the boundary conditions
\begin{align}
 \label{eul:contDualStrongBC}
\big(\matP\Plus(\ww,\nn) \big)\T \zz = 0 \text{ on } \gomIO, \qquad
n_1 \z_2 + n_2 \z_3 = \nn \cdot \vtheta \text{ on } \gomW.
\end{align}
\end{definition}
Obviously, if $\zz$ is the solution in the sense of Definition~\ref{def:z}, then it fulfils
the weak form of the adjoint problem \eqref{eul:weakContDual}.


\section{Discontinuous Galerkin discretization}
\label{sec:eulDG}

In this section, we recall the discontinuous Galerkin (DG) discretization of
the Euler equations \eqref{eul:EE}. Let $\Th=\{K\}$ be a mesh covering $\oOm$
consisting of non-overlapping elements $K\in\Th$. We introduce the
{\em broken Sobolev space} of vector-valued functions
\begin{align}
  \label{bss}
  \bbHj := \{\va\in[L^2(\Om)]^4;\ \va|_K\in [H^1(K)]^4\ \forall K\in\Th\}.
\end{align}
Moreover, we define the interior and exterior traces of $\w\in\bbHj$ on $\dK$,
 $K\in\Th$ by symbols $\w_{\dK}^{(+)}$ and $\w_{\dK}^{(-)}$, respectively,
where 
\begin{align}
  \label{traces}
  \w_{\dK}^{(\pm)}(x) = \lim_{\varepsilon\to 0^+} \w(x \mp \varepsilon \nn),
  \quad
  x\in\dK
\end{align}
and $\nn$ denotes the unit outer normal to $\dK$. Further, we denote
the mean value and the jump of $\w\in\bbHj$ on $\dK\not\subset\gom$ by
\begin{align}
  \label{jump}
  \aver{\w}_{\dK} = (\w_{\dK}^{(+)} + \w_{\dK}^{(-)})/2\qquad \mbox{and}\qquad
  \jump{\w}_{\dK} = \w_{\dK}^{(+)} - \w_{\dK}^{(-)},
\end{align}
respectively. For $\dK\subset\gom$, we put $\aver{\w}_{\dK} = \w_{\dK}^{(+)}$.
In the following, we omit the subscript $_{\dK}$ for simplicity and write
$\aver{\cdot}$ and $\jump{\cdot}$ only.

Finally, similarly as \eqref{defD}--\eqref{spaces0},
we define the space of piece-wise regular admissible
physical state vectors
\begin{align}
  \label{VVh}
  \VVh:=\{\w\in \bbHj;\ \w(x)\in\defD\ \mbox{a.e. } x\in\Om\}.
\end{align}

\subsection{Approximate solution}
The DG approximate solution of \eqref{eul:EE} is sought in
a finite-dimensional subspace of $\bbHj$ which consists of piecewise polynomial
functions. 
We denote the local polynomial degree $p_K \in \matN$ for each $K \in \mesh$ and 
we introduce $\pp := \{ p_K ; K \in \mesh \}.$
Over the triangulation $\mesh$ we define the space
of vector-valued discontinuous piecewise polynomial functions
\begin{align}
  \label{eul:Shp}
  \bShp & =  \{\va_h \in [L^2(\Omega)]^4; \vp_h\vert_{K}\in [P^{p_K}(K)]^4\ \Forall K \in \mesh\}, \\
  \bShpp & =  \{\va_h \in [L^2(\Omega)]^4; \vp_h\vert_{K}\in [P^{p_K+1}(K)]^4\ \Forall K \in \mesh\}, 
  \notag
\end{align}
 
Multiplying \eqref{eul:EE} by $\vvp \in \bbHj,$ integrating over $\Om$ 
and applying the Green theorem separately on each element $K\in \mesh$  and using
\eqref{eul:A_s}--\eqref{eul:Pf},
we get
\begin{align}\label{eul:eq2}
  -  \sum_{K\in\mesh} \int_K  \sum_{s=1}^{2} \matA(\w)\w \cdot \frac{\partial\va}{\partial x_s} \dx 
  &+\sum_{K\in\mesh} \int_{\partial K} \PPP(\ww,\nn) \cdot  \va\dS 
  = 0.
\end{align}
The boundary integrals are approximated by
a {\em numerical flux} 
$\HH: \defD\times \defD \times \Br \to \IR^{4}$ 
\begin{align} \label{eul:nFluxIn}
\intS{\dK}{ \PPP(\ww,\nn) \cdot \va} 
 \approx
\intS{\dK}{\HH(\ww\plus, \w\minus,\nn)\cdot\va }.
\end{align}
We assume that the numerical flux is locally Lipschitz continuous and 
\begin{align}
  \label{H:cons}
  &\mbox{ {consistent}:}\quad   \HH(\w,\w,\nn) = \PPP(\ww,\nn), \quad \w\in\defD,\ \nn\in \Br, \\
  &\mbox{{conservative}:}\quad
  \HH(\w_1,\w_2,\nn) = -\HH(\w_2,\w_1,-\nn), \ \  \w_1, \w_2 \in \defD,\ \nn\in \Br. \notag
\end{align}



\begin{definition}
We say that a function $\w_h \in \bShp$ is 
the {\em approximate DG solution} of the Euler equations \eqref{eul:EE}, if 
\begin{align}
  \label{eul:discPrimal} 
  \bbh(\wh,\va_h) = 0 \qquad \forall \va_h \in \bShp,
\end{align}
where
\begin{align}
  \label{eul:form2}
  \bbh(\wh,\va) & := 
  - \sum_{K\in\mesh}\intx{K}{}{\sum_{s=1}^2 (\matA_s(\wh)\wh) \cdot \frac{\partial \va}{\partial x_s}}{x} \\
  & +\sum_{K \in \mesh} \intS{\partial K }{\HH(\wh\plus, \wh\minus,\nn)\cdot \va}\qquad 
  \wh\in\VVh, \va\in\bbHj. \notag
\end{align}
\end{definition}

We employ the {\em Vijayasundaram numerical flux} \cite{Vijaya}
since it can be easily linearized 
and applied to the setting of the discrete adjoint problem.
Then the numerical flux on $\dK$ is given by
\begin{align}
 \label{eul:VS}
 \fluxVS(\wh\plus,\wh\minus,\nn)=\matP^+\left(\aver{\wh},\nn\right)\wh\plus +
 \matP^-\left(\aver{\wh},\nn\right)\wh\minus,
\end{align}
where $\matP^\pm$ are the positive and negative parts of $\matP$ given by \eqref{Ppm}
and $\aver{\cdot}$ is the means value given by \eqref{jump}.
The numerical flux through the boundary $\gom$ is treated in the next section.
The following manipulations can be simply adopt to, e.g., the Lax-Friedrichs numerical flux.

\subsection{Boundary conditions} 

We discuss two ways of the realization of the impermeable boundary condition \eqref{eul:BCw}
written as $\vv\cdot\nn =0$ on $\gomW$.
The first one is based on the direct use of the impermeability condition
in the physical flux $ \PPP(\w,\nn)$ (\cite{Hartmann2007Adjoint})
and the second applies the so-called {\em mirror operator} to the state $\w$
(e.g., \cite{bas-reb-JCP,harthou2}).
Further, we recall the treatment of the  input/output boundary conditions.
Later we show that all treatments enable deriving the
primal as well as adjoint consistent discretizations.

\subsubsection{Impermeability condition using the boundary value operator}
\label{sec:W1}
We introduce the \emph{boundary value operator}, see \cite{Hartmann2015Generalized},
$\uBC:\defD\to\defD$ on $\gomW$ such that
$\uBC(\wh)$ fulfils  the boundary condition \eqref{eul:BCw}, 
i.e., $\BC(\uBC(\wh)) = 0$, 
and further $\uBC(\ww) = \ww$ for any $\ww$ satisfying $\BC(\ww) = 0$.
Hence, we define $\uBC$ by
\begin{align} 
\label{eul:bvW}
  \uBC (\w):= \matU_{\Gamma} \w =
  \setlength\arraycolsep{5pt}
  \begin{pmatrix} 
  1 & 0 & 0 & 0 \\ 
  0 & 1 - n_1^2 & - n_1 n_2 & 0 \\ 
  0 & - n_1 n_2 & 1 - n_2^2 & 0 \\ 
  0 & 0 & 0 & 1 
 \end{pmatrix} 
 \w \qquad \text{ on } \gomW,
\end{align}
where $(n_1,n_2)$ are components of the unit outer normal $\nn$ to $\gomW$.
Such choice originates in the substracting of the normal component of the velocity,
i.e., $\vv$ is replaced by $\vv - (\nn \cdot \vv) \nn$.
That also obviously guarantees meeting the boundary condition \eqref{eul:BCw}.

Then since $\uBC(\wh) \cdot \nn = 0,$ we  define 
the numerical flux on $\gomW$ by
\begin{align}
  \label{eul:BCw:boundaryOperator}
  \HboundWone( \w\plus, \nn) 
  := \sum_{s=1}^{2} \ff_s(\uBC( \w\plus)) n_s. 
\end{align}
Moreover, due to \eqref{eul:Pf}, \eqref{ppn} and \eqref{lem1.2},
we obtain
\begin{align}
\label{eul:BCw:boundaryOperator2}
\HboundWone( \w\plus, \nn) 
  & = \ppn(\uBC( \w\plus)) 
= \matP_W(\uBC( \w\plus), \nn) \uBC( \w\plus).
\end{align}

\subsubsection{Impermeability condition using the mirror operator}
\label{sec:W2}

We define the {\em mirror operator}
$\Mir:\defD\to\defD$
such that the density, energy and the tangential component of velocity
of $\Mir(\w)$ are the same as of $\w$ and the normal component of velocity has the
opposite sign. Then $\vv$ is replaced by $\vv - 2(\nn \cdot \vv) \nn$ and we put 
\begin{align} 
  \label{Mir}
  \Mir(\w) =  \matM_\Gamma \w :=
  \setlength\arraycolsep{5pt}
  \begin{pmatrix}
    1 & 0 & 0 & 0 \\ 
    0 & 1 - 2 n_1^2 & -2 n_1 n_2 & 0 \\ 
    0 & - 2 n_1 n_2 &1 - 2 n_2^2 & 0 \\ 
    0 & 0 & 0 & 1 
  \end{pmatrix} \w,
\end{align}
where $(n_1,n_2)$ are components of the unit outer normal $\nn$ to $\gomW$.
Obviously, \eqref{eul:bvW} and \eqref{Mir} imply $\Mir(\w) = 2 \uBC(\w) - \w$.


Now we set the numerical flux $\HboundWtwo$ on $\gomW$ as
the Vijayasundaram numerical flux \eqref{eul:VS} with $\w\minus= \Mir(\w\plus),$ i.e., 
\begin{align} \label{eul:mirBC}
  \HboundWtwo(\w\plus,\nn) 
  := \fluxVS(\w\plus, \Mir(\w\plus), \nn),
\end{align}
where the expression $\aver{\w}$ which appears in the definition of $\fluxVS$ is defined as 
$\aver{\w} = \tfrac12(\w\plus + \Mir(\w \plus)) = \uBC(\w\plus)$ on $\gomW.$
    
\begin{remark} 
  Using \eqref{eul:BCw:boundaryOperator2}, \eqref{eq:PPw1},
  identity $\matM_{\Gamma} \matU_{\Gamma} = \matU_{\Gamma}$, \eqref{eul:bvW} and \eqref{Mir},
  we have
  \begin{align} 
    \HboundWone(\w\plus, \nn) &= \matP_W(\uBC( \w\plus)) \matU_{\Gamma} \w \plus  
    = \matP(\uBC( \w\plus)) \matU_{\Gamma} \w \plus \nonumber \\
    &= (\matP\Plus(\uBC( \w\plus)) + \matP\Minus(\uBC( \w\plus)) \matM_{\Gamma}) \matU_{\Gamma} \w \plus \nonumber \\ 
    & =  \fluxVS( \matU_{\Gamma} \w\plus, \Mir( \matU_{\Gamma}\w\plus), \nn). 
  \end{align}
 
\end{remark}

\begin{lemma}
  Boundary numerical fluxes $\HboundWone$ and $\HboundWtwo$ given by
  \eqref{eul:BCw:boundaryOperator} and \eqref{eul:mirBC}, respectively,
  are consistent with the boundary condition \eqref{eul:BCw},
  i.e.,
  \begin{align}
    \label{HB:cons}
    \mbox{if } \w\in\defD \mbox{ satisfy } \eqref{eul:BCw} \ \Rightarrow\ 
    \HboundWi( \w\plus, \nn)= \PPP(\ww,\nn) \mbox{ on }\gomW, 
  \end{align}
  where $i=1,2$ and $\PPP$
  is the physical flux \eqref{eul:Pf}.
\end{lemma}
\begin{proof}
  The consistency of $\HboundWi$, $i=1,2$ follows from the consistency of the
  boundary values operators $\uBC$ and $\Mir$ which means that
  $\uBC(\ww) = \ww$ and $\Mir(\ww) = \ww$ for $\ww$ satisfying \eqref{eul:BCw} on $\gomW$.
  \qed
\end{proof}


\subsubsection{Boundary conditions on the inlet/outlet}
We describe the realization of the inlet/outlet boundary condition \eqref{eul:BCio}.
For simplicity, we restrict to the flow around an isolated profile where
the state vector $\wBC = \ww_{\infty}$ stands for the free-stream free flow.
We define the vector $\wh\minus$ as the solution of the local linearized Riemann problem
with states $\wh\plus$ and $\wBC$, whose solution can be written
as
\begin{align}
  \label{LLRP}
  \uuRP(\wh\plus, \wBC) = \matP\Plus(\wh\plus,\nn) \wh\plus + \matP\Minus(\wh\plus,\nn) \wBC,
\end{align}
where $\nn$ is unit outer normal to $\gomIO$ and $\matP^{\pm}$ are given by \eqref{Ppm}.
We refer to, e.g., \cite{feikuc2007} or \cite[Chapter~8]{DGM-book}. 
Finally, we put $\wh\minus:=\uuRP(\wh\plus, \wBC)$ and 
\begin{align} 
\label{eul:Hio}
\HboundIO(\wh\plus, \nn) := \fluxVS(\wh\plus, \uuRP(\wh\plus, \w_{\rm BC}), \nn),
\end{align}
where we set $\aver{\wh} := \wh\plus$ for any $\wh \in \bbHj$ on $\gomIO $ in \eqref{eul:VS}.

\begin{lemma}
  The numerical flux $\HboundIO$ given by \eqref{eul:Hio} is consistent with \eqref{eul:BCio}, i.e.,
   \begin{align}
    \label{HIO:cons}
    \mbox{if } \w\in\defD \mbox{ satisfy } \eqref{eul:BCio} \ \Rightarrow\ 
    \HboundIO( \w\plus, \nn)= \PPP(\ww,\nn) \mbox{ on }\gomIO. 
  \end{align} 
\end{lemma}
\begin{proof}
  If $\w$ fulfils \eqref{eul:BCio} then
  $\matP\Minus(\w\plus,\nn)\w\plus =\matP\Minus(\w\plus,\nn) \wBC$
  and
  \begin{align*}
    \uuRP(\w\plus, \wBC) =  \matP\Plus(\w\plus,\nn) \w\plus + \matP\Minus(\w\plus,\nn) \w\plus
    = \matP(\w\plus,\nn) \w\plus,
  \end{align*}
  which together with \eqref{PPP} implies \eqref{HIO:cons}. \qed
\end{proof}

In the following the symbol $\Hbound$ stands either for $\HboundWi$, $i=1,2$ or
$\HboundIO$. In the context is clear which numerical flux is used.

\subsection{Primal consistency}

In this section, we prove the consistency of the DG discretization, i.e., if $\ww$ is the
smooth solution of \eqref{eul:EE} then it fulfils the identity
\eqref{eul:discPrimal} with \eqref{eul:form2} and the corresponding boundary numerical fluxes
\eqref{eul:BCw:boundaryOperator}, \eqref{eul:mirBC} and \eqref{eul:Hio}.
We define the primal residual of problem \eqref{eul:discPrimal} by
\begin{align}
 \label{eul:primalRes} 
 \resEul{\vvp} := - \bbh(\wh, \vvp),\qquad \wh\in\VVh, \vvp\in \bbHj.
\end{align}
Integrating \eqref{eul:form2} by parts on each $K\in\Th$
and using \eqref{eul:A_s}--\eqref{eul:Pf}, we get
\begin{align}
  \label{eul:formRes1}
   \resEul{\va_h} := & - \sum_{K\in\mesh} \left( \intx{K}{}{\sum_{s=1}^2  \frac{\pd \ff_s(\wh)}{\pd x_s} \cdot \va_h}{x} \right. \\
   & \left. +  \intS{\partial K \backslash \dom}{ \left(\PPP(\wh\plus,\nn) - \HH(\wh\plus, \wh\minus,\nn) \right) \cdot \va_h\plus} \right. \nonumber \\ 
   & \left. 
   + \intS{\partial K \cap \dom}{ \left(\PPP(\wh\plus,\nn) - \Hbound(\wh\plus, \nn)  \right) \cdot \va_h} \right). \nonumber
\end{align}
Due to \eqref{PPP}, \eqref{LLRP}--\eqref{eul:Hio},
the integrand of the last term of \eqref{eul:formRes1}
on $\gomIO$ reads, 
\begin{align} 
 \label{eul:resIO}
   &  \PPP(\wh\plus,\nn) - \Hbound(\wh\plus,\nn)   \\
  = & 
  \matP(\wh\plus,\nn)\wh\plus - \matP\Plus(\wh\plus,\nn)\wh\plus
 - \matP\Minus(\wh\plus,\nn ) \uu_\mathrm{RP}(\wh\plus,\wBC) \nonumber \\
 = & \matP\Minus(\wh\plus,\nn)( \wh\plus - \uu_\mathrm{RP}(\wh\plus,\wBC) ). \nonumber
\end{align}
Based on \eqref{eul:formRes1}--\eqref{eul:resIO}, we define for any $K \in \mesh$
the element primal residuals
\begin{align} 
  \label{eul:primalResiduals1}
  \resFK{\wwh} & := - \sum_{s=1}^2 \frac{\partial}{\partial x_s} \ff_s(\wwh) 
  = - \sum_{s=1}^2 \matA_s(\wwh) \frac{\partial \wwh}{\partial x_s} \qquad \text{in } K,\\ 
  \resFE{\wwh} & := 
  \begin{cases}
    \PPP(\wh\plus,\nn) - \HH(\wh\plus, \wh\minus,\nn)  \qquad &\text{on } \partial K \backslash \gom,\\
    \noalign{\vskip2pt}
    \matP\Minus(\wh\plus,\nn)\left( \wh\plus - \uu_\mathrm{RP}(\wh\plus,\wBC)\right) \qquad &\text{on } \partial K \cap \gomIO,\\
    \noalign{\vskip2pt}
    \PPP(\wh\plus,\nn) - \HboundW(\wh\plus, \nn)  \qquad &\text{on } \partial K \cap \gomW,\\
  \end{cases}
  \notag
\end{align}
where the term $\HboundW(\wh\plus,\nn)$ stands for either
\begin{align*}
  \matP_W(\uBC(\wh\plus),\nn) \uBC(\wh\plus) \quad \text{or} \quad
  \fluxVS(\wh\plus,\Mir(\wh\plus), \nn),
\end{align*}
depending on whether $\HboundWone$ or $\HboundWtwo$ is used, 
cf. \eqref{eul:BCw:boundaryOperator} and \eqref{eul:mirBC}, respectively.

Employing \eqref{eul:primalResiduals1} in \eqref{eul:formRes1} we obtain the residual form of the problem \eqref{eul:discPrimal}: 
find $\wwh \in \bShp$ such that
\begin{align}
  \label{eul:formRes2}
   &\resEul{\va_h} = \sum_{K\in\mesh} \left( \intx{K}{}{\resFK{\wwh} \cdot \va_h}{x}
   +  \intS{\partial K}{ \resFE{\wwh} \cdot \va_h\plus} 
   \right)
   = 0 
\end{align}
holds for any $\va_h\in \bShp$.
The previous derivation leads to the following result.
\begin{theorem}
Let the numerical fluxes $\HH$ and $\Hbound$ used on inner and boundary edges 
be consistent (cf. \eqref{H:cons}, \eqref{HB:cons} and \eqref{HIO:cons}), 
then the discretization \eqref{eul:discPrimal} is consistent, i.e., 
if $\ww \in \VV$ is the exact solution of \eqref{eul:EE} then
it also nullifies the discrete formulation \eqref{eul:discPrimal}:
\begin{align} 
 \resEulTrue{\va} = 0 \qquad \forall \va \in \bbHj.
\end{align}
\end{theorem}



\section{Discrete adjoint problem and the adjoint consistency} 
\label{sec:EulerNewton}

The discrete problem \eqref{eul:discPrimal} exhibits the system of nonlinear algebraic equations
which has to be solved iteratively.
The popular Newton method (treated, e.g., in \cite{bas-reb-JCP,BR-2DGM,HH06:SIPG1}) requires
the evaluation of the Jacobi matrix.
However, the terms corresponding to the numerical Vijayasundaram fluxes \eqref{eul:VS} are not
continuously differentiable and then a regularization would be required.
Therefore we do not compute the derivative $\ahP[\uh](\cdot,\cdot)$ precisely, 
but instead we approximate it by the linearized form
\begin{align}
  \label{lin}
 \ahP[\wh](\cdot,\cdot) \approx \ahL(\wh,\cdot,\cdot),
\end{align}
which we employed in \cite{impl_eu,st_estims_NS,stdgm_est}. However, in these papers,
we considered a different treatment of wall boundary conditions, which leads to
a non-adjoint consistent discretization. 
In the following, we present a modified linearization of type
\eqref{lin} which is adjoint consistent.

\subsection{Linearization of the form $\bbh$} 

The semilinear form \eqref{eul:form2} can be written as 
\begin{align}
  \bbh(\wh,\vah) = 
  & - \sum_{K\in\mesh}\intx{K}{}{\sum_{s=1}^2
    ( \matA_s(\wh) \wh )\cdot \frac{\partial \vah}{\partial x_s}}{x}
  &\hspace{-3pt} (=:\chij(\wh,\vah))  \\
  &+ \sum_{K \in \mesh} \intS{\partial K \backslash \dom}{\HH(\wh\plus, \wh\minus,\nn)\cdot \vah}
  &\hspace{-3pt}(=:\chid(\wh,\vah)) \nonumber\\
  & +\sum_{K \in \mesh} \intS{\partial K \cap \gomW}{ \HboundWi(\wh\plus, \nn) \cdot \vah}
  &\hspace{-3pt}(=:\chit(\wh,\vah)) \nonumber\\
  &  +\sum_{K \in \mesh} \intS{\partial K \cap \gomIO}{\HboundIO(\wh\plus, \nn) \cdot \vah }
  &\hspace{-2pt}(=:\chic(\wh,\vah))  \notag
\end{align}
with $i=1,2$. We linearize each of the four terms $\chij,\ldots,\chic$ individually.

For the first one we define the linearized form $\chijL : \VVh \times \bbHj \times \bbHj \to \matR$ 
by
\begin{align}
 \label{eul:chi1}
 \chijL(\whb,\wh, \va_h) = - \sum_{K\in\mesh}\intx{K}{}{\sum_{s=1}^2 \matA_s(\whb) \wh \cdot \frac{\partial \va}{\partial x_s}}{x}.
\end{align}
Employing \eqref{eul:A_s} we have
$\chijL(\wh,\wh, \va_h) = \chij(\wh,\va_h)\ \forall \wh\in\VVh\,\forall\va_h\in\bbHj$
and obviously $\chijL$ is linear with respect to its second and third arguments. 

For linearization of the term $\chid$ we exploit the definition of the Vijayasundaram numerical fluxes \eqref{eul:VS}.
Since every inner edge in the triangulation appears twice in the sum 
we reorganize the summation.
Using the notation \eqref{jump},
the linearized form $\chidL:\VVh \times \bbHj \times \bbHj\to\matR$ reads
\begin{align}
 \label{eul:chi2}
  \chidL(\whb, \wh,\va_h) 
 &= \sum_{K \in \mesh} \intS{\partial K \backslash \dom}{ \big[\matP^+\big(\aver{\whb}_{\Gamma},\nn \big) \wh \plus \\
 & \qquad \qquad \qquad + \matP^-\big(\aver{\whb}_{\Gamma},\nn \big) \wh\minus \big]\cdot \va_h \plus } \nonumber\\ 
  &= \sum_{K \in \mesh} \intS{\partial K \backslash \dom}{ \matP^+\left(\aver{\whb}_{\Gamma},\nn \right) \wh \plus \cdot \jump{\va_h}_{K} }.
  \nonumber
\end{align}
Obviously $\chidL(\wh,\wh, \va_h) = \chid(\wh,\va_h)\ \forall \wh\in\VVh\,\forall\va_h\in\bbHj$
and $\chidL$ is linear with respect to its second and third arguments. 

Regarding the term $\chit$ we have to proceed separately for each of the approaches $\HboundWi, \, i = 1,2$. 
Based on the definition \eqref{eul:BCw:boundaryOperator} of $\HboundWone$
we may introduce its linearization in the following form
  \begin{align} \label{eul:fWFLTwo}
  \HboundLOne(\bar{\w},\w,\nn) 
  = \matP_W(\uBC(\bar{\w}),\nn)\matU_{\Gamma} \w,\quad \bar{\w}, \w\in\defD. 
  \end{align}
The linearization of $\HboundWtwo$  is done similarly to \eqref{eul:chi2}. 
Since 
$\frac{\w + \Mir(\w)}{2} = \uBC(\w)$,
in virtue of \eqref{eul:mirBC} and \eqref{eul:VS}, we get
\begin{align} \label{eul:fluxThreeLin}
 \HboundLTwo(\bar{\w},\w,\nn) = \left( \matP^+\left(\uBC(\bar{\w}),\nn \right) +
\matP^-\left(\uBC(\bar{\w}),\nn \right) \matM_{\Gamma} \right) \w\plus.
\end{align}
Employing the linearized forms \eqref{eul:fWFLTwo} and \eqref{eul:fluxThreeLin}, we set
\begin{align} 
  \label{eul:chi3one}
  \chitiL(\whb, \wh,\va_h) &= \sum_{K \in \mesh} \intS{\partial K \cap \gomW}{ \HboundLI(\whb,\wh,\nn) \cdot \va_h} \\
							&= \sum_{K \in \mesh} \intS{\partial K \cap \gomW}{\va\T_h \HboundMatrix{i}(\whb,\nn) \wh}, \nonumber
\end{align}
where $i=1,2$ and the matrix $\HboundMatrix{i}(\whb,\nn)$ corresponds to 
one of the matrices in \eqref{eul:fWFLTwo} and \eqref{eul:fluxThreeLin}, i.e., 
\begin{align}
\label{eul:Hmatrix1}
\HboundMatrix{1}(\whb,\nn) &= \matP_W(\uBC(\bar{\w}),\nn)\matU_{\Gamma}, \\
\label{eul:Hmatrix2}
\HboundMatrix{2}(\whb,\nn) &= \matP^+\left(\uBC(\bar{\w}),\nn \right) +
\matP^-\left(\uBC(\bar{\w}),\nn \right) \matM_{\Gamma}.
\end{align}
By exploring the definitions of $\HboundLI, \, i = 1,2$ we get that both $\chitiL$ 
are linear with respect to the second and third argument 
and they meet the consistency property
$\chitiL(\wh, \wh,\va_h) = \chiti(\wh,\va_h)\ \forall \wh\in\VVh\,\forall\va_h\in\bbHj.$

At last, $\chic$ is approximated with the aid of the forms 
\begin{align}
 \label{eul:chi4}
  \chicL(\whb, \wh,\va_h)
  =  \sum_{K \in \mesh} \intS{\partial K \cap \gomIO}{\left(\matP^+(\whb\plus,\nn) \wh\plus \right)\cdot\va_h},
\end{align}
and
\begin{align}
\label{eul:bbht}
\bbht(\bar{\w}_h,\va_h)
 =   - \sum_{K \in \mesh} \intS{\partial K \cap \gomIO}{\left(\matP^-(\whb\plus,\nn) 
       \whb\minus \right) \cdot\va_h},
\end{align}
where  $\whb\minus=\uuRP(\whb\plus, \wBC)$, cf. \eqref{LLRP}.
Let us underline that in the arguments of $\matP^\pm$ we use just $\whb\plus$ and
not the mean value of the left- and right-hand side state vectors as in \eqref{eul:VS}.
Moreover, if $\mbox{supp}\, \va_h\cap(\gomIO) =\emptyset$,
then $\bbht(\bar{\w}_h,\va_h)=0$.
 
Obviously, due to \eqref{eul:chi4} and \eqref{eul:bbht}, we have
\begin{align}
 \label{chiz}
 \chicL(\w_h, \w_h,\va_h)   - \bbht(\w_h,\va_h) = \chic(\w_h,\va_h)\qquad
 \forall \w_h\in\VVh\,\forall\va_h\in \bbHj.
\end{align} 

Taking together all the previously defined linearizations, 
we set 
\begin{align}
  \label{eul:bbhL}
  &\ahL( \whb, \wh, \va_h) = \sum_{i=1}^4 \chiLi( \whb, \wh, \va_h),\qquad
  \whb, \wh\in\VVh\,\forall\va_h\in \bbHj
\end{align}
and we get the consistency relation
\begin{align}
  \label{consA}
  \bbh(\wh,\va_h) = \ahL( \wh, \wh, \va_h) -  \bbht(\wh,\va_h)\qquad \forall
  \wh\in\VVh\,\forall\va_h\in\bbHj.
\end{align}

Finally, we introduce the iterative process for the solution of \eqref{eul:discPrimal}.
Let $\whN\in\bShp$ be an initial approximation, 
we define the sequence $\whk\in\bShp,\ k=1,2,\dots$ such that
\begin{subequations}
  \label{newton}
  \begin{align}
  \label{newton1}
  & \whkP := \whk + \delta \dhk, \quad k=0,1,\dots\\
  \label{newton2}
  & \mbox{where } \dhk\in\bShp \mbox{ solves } \quad
   \ahL( \whk, \dhk, \va_h) =  - \bbh(\whk,\va_h) \quad \forall \va_h\in\bShp, 
  \end{align}
\end{subequations}
and $\delta\in(0,1]$ is the damping factor improving the global convergence.
  The identity \eqref{newton2} exhibits a linear algebraic
  system which is solved iteratively by, e.g., GMRES method with block ILU(0) preconditioner,
  see \cite{stdgm_est} for details. 
  For $\delta=1$, the iterative process \eqref{newton}
  is equivalent to
  \begin{align*}
    \ahL( \whk, \whkP, \va_h) =  \bbht(\whk,\va_h)\quad \forall\va_h\in\bShp,\ k=0,1,\dots.
  \end{align*}

\subsection{Discrete adjoint problem and adjoint consistency} 
\label{sec:eulAdjointConsistency}

In this section we introduce the discrete adjoint problem based on the linearization of
the form $\bbh$ 
given by \eqref{eul:bbhL}. Further, the adjoint consistency of the discretization is studied. 

In order to obtain an adjoint consistent scheme, we modify the target functional $\J$ from
\eqref{eul:cDL} as generally mentioned in the introduction. 
For the functional given by \eqref{eul:cDL} we set 
\begin{align} 
\label{eul:Jd}
 \Jd(\wh) = \intS{\gomW}{ \HboundWi(\wh\plus,  \nn) \cdot \tvtheta},\quad i=1,2,
\end{align}
where $\HboundWone$ and $\HboundWtwo$ are given by \eqref{eul:BCw:boundaryOperator}
and \eqref{eul:mirBC}, respectively, 
and $\tvtheta = (0,\vartheta_1, \vartheta_2, 0)\T$ on $\gomW$ with $\vtheta$ 
chosen either by \eqref{eul:thetaDL} or \eqref{eul:thetaM}.
Obviously, if $\w$ is the 
exact solution of \eqref{eul:EE} then, due to \eqref{HB:cons} and \eqref{eq:PPw2}, we have
\begin{align} 
 \HboundWi(\ww\plus, \nn) \cdot \tvtheta = \press(\w) \nn \cdot \vtheta,\quad i=1,2.
\end{align}
By comparison of the definitions \eqref{eul:cDL} and \eqref{eul:Jd}, we observe that
\begin{align}
  \label{consJJh}
  \Jd(\w) = \J(\w),
\end{align}
which means that the particular modification $\Jd$ is consistent with $\J$.
Further, using the linearization of the numerical fluxes  \eqref{eul:fWFLTwo} and \eqref{eul:fluxThreeLin}, 
we introduce the linearization of the discrete functional 
\begin{align} 
\label{eul:JdLin}
 \JdL(\wh, \vah) & = \intS{\gomW}{ \HboundLI(\wh, \vah, \nn) \cdot \tvtheta} \\
 & = \intS{\gomW}{\vahT \Big(\HboundMatrix{i}(\wh\plus,\nn))\Big)\T  \tvtheta},\quad  i=1,2,\notag
\end{align}
where $\HboundMatrix{i},\ i=1,2$ are given by \eqref{eul:Hmatrix1}--\eqref{eul:Hmatrix2}.

 Finally we introduce the discrete adjoint problem.
\begin{definition} 
 We say that $\zzh \in \bShp$ is the discrete adjoint solution 
 if it satisfies
\begin{align} 
\label{eul:discrAdjoint}
\ahL( \wh, \va_h, \zzh) = \JdL(\wh,\va_h) \qquad \forall \va_h \in \bShp,
\end{align}
where $\ahL$ and $\JdL$ are given by \eqref{eul:bbhL} and \eqref{eul:JdLin}, respectively.
Further we define the adjoint residual 
\begin{align}
 \label{eul:adjointRes} 
 \resEulD{\va_h} :=  \JdL(\wh,\va_h) - \ahL( \wh, \va_h, \zzh),\ \wh\in\VVh,\,\va_h,\zzh\in\bbHj.
\end{align}
\end{definition}

\begin{theorem}
  \label{the:dual_cons}
  Let $\fluxVS$ be  the Vijayasundaram numerical flux. 
 Let $\Jd$ be the modified target functional defined in \eqref{eul:Jd}.
 Then the discretization \eqref{eul:discPrimal} is adjoint consistent, i.e., 
 the exact solution $\w$ of the flow equations \eqref{eul:EE} and its adjoint counter-part $\zz,$ solving the continuous adjoint problem \eqref{eul:weakContDual}, satisfy
 \begin{align} 
 \label{eul:adjointConsistency} 
  \resEulDTrue{\va} = 0 \qquad \forall \va \in \Vp.
 \end{align}
\end{theorem}
\begin{proof}
  Similarly as the residual formulation of the primal problem \eqref{eul:formRes2},
  we introduce,   using \eqref{eul:adjointRes},
  the residual formulation of the discrete problem \eqref{eul:discrAdjoint} by 
\begin{align} 
 \label{eul:discrAdjoint2} 
 &\resEulD{\va_h} \\
 =& \sum_{K \in \mesh} \intx{K}{}{ \resDFK{\wwh}{\zzh} \cdot \va_h}{x} + \intS{\dK}{ \resDFI{\wwh}{\zzh} \cdot \va_h\plus} 
 = 0 \quad \forall \va_h \in \bShp, \notag
\end{align}
where the volume and edge residual terms are defined by 
\begin{align} 
\label{eul:dualResiduals1}
 \resDFK{\wwh}{\zzh} &=  \sum_{s=1}^2  \matA_s\T(\wh) \frac{\partial \zzh}{ \partial x_s}  \qquad \text{in } K,\\ 
 \label{eul:dualResiduals2}
 \resDFI{\wwh}{\zzh} &= 
 \begin{cases}
- \matP\Plus(\aver{\wwh},\nn)\T \jump{\zzh} \qquad &\text{on } \partial K \backslash \gom, \\ 
    \noalign{\vskip2pt}
 - \matP\Plus(\wh\plus,\nn)\T \zzh \qquad &\text{on } \partial K \cap \gomIO, \\
    \noalign{\vskip2pt}
 (\HboundMatrix{i}(\wh\plus,\nn))\T ( \tvtheta - \zzh) ,\ i=1,2 &\text{on } \partial K \cap \Gamma_{W},
 \end{cases}
\end{align} 
which follows from the definitions of $\chiLi, i=1,\ldots,4$ in 
\eqref{eul:chi1}, \eqref{eul:chi2}, \eqref{eul:chi3one}, \eqref{eul:chi4} and 
the definition of the linearization of the modified target functional \eqref{eul:JdLin}.

 Employing  \eqref{eul:discrAdjoint2} -- \eqref{eul:dualResiduals2}, 
 we  rewrite the left-hand side of \eqref{eul:adjointConsistency} to 
 \begin{align} 
 \sum_{K \in \mesh} \intx{K}{}{ \resDFK{\ww}{\zz} \cdot \va}{x} + \intS{\dK}{ \resDFI{\ww}{\zz} \cdot \va\plus} 
 \qquad \forall \va \in \Vp.
\end{align}
Reminding the strong formulation of the continuous adjoint problem \eqref{eul:contDualStrong} we 
see that $\resDFK{\ww}{\zz} = 0$  for any $K \in \mesh.$
Further, due to the assumed smoothness of the adjoint solution $\zz$ we also have 
$\resDFI{\ww}{\zz} = 0$ on $\dKI.$


The residuals on the boundary $\gom$ are examinated separately.
If the numerical flux $\HboundWone$ given by \eqref{eul:BCw:boundaryOperator}
is used on $\gomW$, we 
exploit that $\uBC(\ww) = \ww$ for the exact solution and $\uBC$ given by \eqref{eul:bvW}.
Recalling \eqref{eul:Hmatrix1} and \eqref{lem1.1}, we get
\begin{align}
  \label{Du8}
\resDFI{\ww}{\zz} 
& = \left(\HboundMatrix{1}(\ww,\nn)\right)\T(\tvtheta - \zz) 
= \matU_{\Gamma}\T \matP\T_W(\uBC(\ww),\nn) (\tvtheta - \zz)  \\
&= \matU_{\Gamma}\T \frac{D \press(\ww)}{D\ww}(0,n_1, n_2, 0)\T  \cdot \left( \tvtheta - \zz \right)  \nonumber \\
&= \matU_{\Gamma}\T \frac{D \press(\ww)}{D\ww} \left( \nn \cdot \vtheta - (n_1 z_2+n_2 z_3) \right) = 0, \nonumber
\end{align}
since the adjoint solution $\zz$ satisfies the boundary condition \eqref{eul:contDualStrongBC}. 

If the numerical flux $\HboundWtwo$ given by \eqref{eul:mirBC} is used on $\gomW$ then
using \eqref{eul:Hmatrix2}, we have 
\begin{align} 
\label{eul:adjointW1}
& \intS{\dK \cap \gomW}{\resDFI{\ww}{\zz} \va}
 =  \intS{\dK \cap \gomW}{\big(\HboundMatrix{2}(\ww,\nn)\big)\T(\tvtheta - \zz) \va} \\
& = \intS{\dK \cap \gomW}
       { (\tvtheta - \zz )\T \big(\matP\Plus(\uBC(\ww),\nn) 
+ \matP\Minus(\uBC(\ww),\nn)\matM_{\Gamma}\big) \va }. \nonumber
\end{align}
Since $\va \in \Vp$, cf. \eqref{eul:Vp}, it holds
$n_1 \va_2 + n_2 \va_3 =  0$ and 
hence $\matM_{\Gamma} \va = \va$.
Further, the exact solution $\w$ satisfies $\vv \cdot \bkn = 0$ then $\uBC(\w) = \w$
and together with  relation $\matP=\matP\Plus+\matP\Minus$ and \eqref{eq:PPw1},
we obtain from \eqref{eul:adjointW1} that 
\begin{align} 
 \intS{\dK \cap \gomW}{\resDFI{\ww}{\zz}}
= & \intS{\dK \cap \gomW}{( \tvtheta - \zz )\T \matP(\ww,\nn) \va } \\
= & \intS{\dK \cap \gomW}{( \tvtheta - \zz )\T \matP_W(\ww,\nn) \va }  \nonumber \\
= & \intS{\dK \cap \gomW}{\va\T \matP_W\T(\ww,\nn) ( \tvtheta - \zz ) } = 0, \nonumber
\end{align} 
where the last equality follows from the same manipulations as in \eqref{Du8}. 

Finally, $\resDFI{\ww}{\zz} = 0$ on $\gomIO$ since $\zz$ fulfils
condition \eqref{eul:contDualStrongBC}. \qed
\end{proof}

\begin{remark}
  Theorem~\ref{the:dual_cons} asserts the adjoint consistency of both treatments
  of impermeable boundary conditions presented in Sections~\ref{sec:W1} and \ref{sec:W2}.
  On the other hand, the discussion at the end of \cite[Section 5]{Hartmann2007Adjoint} implies
  that the treatment of the impermeability condition using the mirror operator
  (sf. Section~\ref{sec:W2}) is not  adjoint consistent. However, it is not in a contradiction 
  with our results since we employ a different definition of the modified functional $\Jd$
  in \eqref{eul:Jd} for the ``mirror'' boundary conditions.
\end{remark}

\begin{remark}
  Let us shortly discuss the pertinence of the discretization \eqref{eul:discrAdjoint}
  of the adjoint problem \eqref{eul:contDualStrong}--\eqref{eul:contDualStrongBC}.
  The discrete formulation \eqref{eul:discrAdjoint} is based on linearization 
  rather than on proper differentiation of the nonlinear discrete problem 
  \eqref{eul:discPrimal} 
  like it is usually done, cf. \cite{Hartmann2015Generalized} or \cite{Hartmann2006Derivation}. 
  On the other hand, the omitted terms contain derivatives of the numerical fluxes \eqref{eul:VS}
  which lack the required smoothness to be differentiated exactly. 
  In \cite{Hartmann2005Role} these terms are approximated by finite differences for
  Lax-Friedrichs and Vijayasundaram numerical fluxes. 
  We note that omitting those terms does not cause any source of inconsistency
  into the discrete problem and from point of view it nicely corresponds to the continuous
  formulation of the adjoint problem \eqref{eul:contDualStrong}, 
  and hence the discretization \eqref{eul:discrAdjoint} seems 
  as a quite reasonable DG discretization of problem
  \eqref{eul:contDualStrong}--\eqref{eul:contDualStrongBC}.
\end{remark}


\section{Error estimates and mesh adaptivity} 
\label{sec:EE}

\subsection{Goal-oriented error estimates}
As mentioned above, 
the adjoint problem is defined usually using  the derivatives of the discrete form $\ahP[\wh]$
and the target functional $\JP[\wh]$.
Then it can be proved that the  error of the quantity of interest is given by
(see, e.g., \cite{RannacherBook,BeckerRannacher01})
\begin{align} 
  \label{J_estim}
  J(\ww) - J(\wh) = \frac{1}{2} \resEul{\zz - \va_h}  + \frac{1}{2} \resEulD{\ww -\vpsi_h} + \Rtri, 
\end{align}
where $\va_h, \vpsi_h \in \bShp$ are arbitrary, 
$\resEul{\cdot}$ and  $\resEulD{\cdot}$ are the residuals of the primal and adjoint
problems similar to \eqref{eul:formRes2} and \eqref{eul:adjointRes}, respectively, and 
$\Rtri= O((\ww-\wh)^3)$ is a higher order term which is neglected.

In the presented formulation of the adjoint problem \eqref{eul:discrAdjoint},
we replaced the derivatives $\bbhP[\cdot]$ and $\JP[\cdot]$  by the linearizations
$\ahL$ and $\JdL$ given by \eqref{eul:JdLin} and \eqref{eul:bbhL}, respectively. 
That  may lead to additional errors, 
but we omit them in the error estimates similarly 
as the term $\Rtri$ is usually omitted even for exactly differentiated schemes. 
The numerical experiments, presented in Section~\ref{sec:numer}, 
indicate that this source of errors does not notably change the estimates 
(compared to results published for similar numerical experiments,
in \cite{Hartmann2007Adjoint}, \cite{Hartmann2015Generalized}).

The error identity \eqref{J_estim} contains the exact primal and adjoint solutions
$\ww$ and $\zz$ which have to be replaced by some computable
higher-order approximations denoted here $\whp$ and $\zhp$, respectively. 
Those can be computed either globally -- on a finer mesh and/or using polynomials
of higher degree, or with local reconstructions. Here, we are using the latter case,
see Section~\ref{sec:nonlReconstruct}.
Then we define the approximation of the error of the quantity of interest  by
\begin{align} 
  \label{eul:etaI}
  J(\w) - J(\wh) &\approx \etaI(\wwh,\zzh) \\
  &:= \tfrac{1}{2}\left( \resEul{\zhp - \Pi \zhp}   + \resEulD{\whp - \Pi \whp} \right), \nonumber
\end{align}
where $\Pi: [L^2(\Om)]^m \to \bShp$ denotes an arbitrary projection on $\bShp.$

For the purpose of mesh adaptation, we rewrite estimate \eqref{eul:etaI} element-wise
\begin{align}
 \label{eul:residuals}
 \etaI(\wwh,\zzh) = \sum_{K\in \mesh} \etaIK,
\end{align}
where 
\begin{align} 
 \label{eul:etaIK}
 \etaIK 
 = \tfrac{1}{2}\left( 
 \resEul{(\zhp - \Pi \zhp) \,\chi_K } + \resEulD{(\whp - \Pi \zhp)\,\chi_K }
 \right).
\end{align}
Here, $\chi_K$ denotes the characteristic functions of the mesh element $K\in\Th$. 
For the mesh adaptation, the values $|\etaIK|$, $K\in\Th$ are used.
It would be possible to replace $\etaI$ by the sum of the absolute values
of the local indicators $\sum_{K\in\Th}|\etaIK|$.
However, this estimate leads usually to a  needless overestimation of the true error
$|J(\w) - J(\wh)|$. 
Finally, let us note that we neglect the errors arising from the solution of nonlinear
algebraic systems by an iterative solver. These additional source of errors will be treated
in a separate paper.




\subsection{Reconstruction based on solving local nonlinear problems} 
\label{sec:nonlReconstruct}

The higher-order approximation $\whp$ and $\zhp$
appearing \eqref{eul:etaI}--\eqref{eul:etaIK} are obtained using
a reconstruction $\krr:\bShp\to\bShpp$, cf. \eqref{eul:Shp}, as
$\whp=\krr(\wh)$ and $\zhp=\krr(\zh)$. The operator $\krr$ is defined by
a solution of local problems.
This technique was derived in \cite{ESCO-18}
for a linear scalar problem and it can be applied to
the reconstruction of the adjoint discrete solution $\zh$ since the adjoint problem is linear.
The situation is a bit different for the reconstruction of $\wh$ 
due to the nonlinearity of the problem \eqref{eul:discPrimal}. 

Similarly to \cite{ESCO-18}, for each $K \in \mesh,$ we prescribe
$\wKp: \Om \to \matR^4$ satisfying: 
\begin{subequations}
    \label{eul:rec}
    \begin{align}
       \label{eul:rec1}
       \mbox{(i)} &\quad \wKp|_{K'} := \wh|_{K'} \mbox{ for all }  K'\in\Th,\ K'\not= K, \\
       \label{eul:rec2}
       \mbox{(ii)} & \quad \wKp|_{K} \in [P^{p_K+1}(K)]^4, \\
       \label{eul:rec3}
       \mbox{(iii)}&\quad \bbh(\wKp, \va_h) = 0 \quad \forall \va_h \in [P^{p_K+1}(K)]^4,
    \end{align}
\end{subequations}
where $\bbh$ is the form given by \eqref{eul:form2}. 
Finally, we define $\whp \in \bShpp$ by  $\whp|_K := \wKp \quad \forall K\in\Th. $ 
The problem \eqref{eul:rec3} is nonlinear we calculate the reconstruction 
$\wKp$ iteratively similarly as the global problem mentioned in \eqref{newton}.
For completeness, let us mention that $\zhp$ is defined similarly as $\whp$ in \eqref{eul:rec}
where we replace \eqref{eul:rec3} by $\ahL(\wh,\vpsi_h,\zKp) = \JdL(\wh,\vpsi_h)$ 
$\forall \vpsi_h \in [P^{p_K+1}(K)]^4$, cf. \eqref{eul:discrAdjoint}.

\subsection{Adjoint weighted residual error estimate}

In order to proceed to the goal-oriented mesh adaptation, 
we estimate the residuals $\resEul{\cdot}$ and $\resEulD{\cdot}$ of the 
primal problem \eqref{eul:discPrimal} and adjoint problem \eqref{eul:discrAdjoint}, respectively. 
Employing the integration by parts like in \eqref{eul:formRes1} and \eqref{eul:formRes2}
the element-wise primal residual can be further estimated by
\begin{align}
  \label{resP}
 \resEul{\va} &= \sum_{K\in\mesh} \bigg( \intx{K}{}{\resFK{\wwh} \cdot \va}{x}
   +  \intS{\dK}{ \resFE{\wwh} \cdot \va\plus} 
   \bigg) 
   \\
   &\leq  \sum_{K\in\mesh} \bigg( \sum_{i=1}^4 \RKVi \norm{\va^i}{K}
   + \RKBi \norm{\va^i}{\dK} \bigg), \nonumber
\end{align}
where $ \RKVi := \normS{\resFKi{\wwh}}{K}$, $\RKBi := \normS{\resFEi{\wwh}}{\dK}$,
the terms $\resFKi{\wwh}$ and $\resFEi{\wwh}$ denote the $i$-th component, 
$i=1,\ldots,4,$ of the local residual terms given by \eqref{eul:primalResiduals1}
and $\va^i$ denotes the $i$-th component of the vector function $\va.$

 
Similarly, we may proceed for the adjoint residual
\begin{align} 
  \label{resD}
 \resEulD{\va} 
 &= \sum_{K\in\mesh} \left( \intx{K}{}{\resDFK{\wwh}{\zzh} \cdot \va}{x}
   +  \intS{\dK}{ \resDFI{\wwh}{\zzh} \cdot \va\plus} 
   \right) \notag 
   \\
   &\leq  \sum_{K\in\mesh} \left( \sum_{i=1}^4 \RDKVi \norm{\va^i}{K}
   + \RDKBi \norm{\va^i}{\dK} \right) 
\end{align}
where $\RDKVi := \normS{\resDFKi{\wwh}{\zzh}}{K}$, 
$\RDKBi := \normS{\resDFIi{\wwh}{\zzh}}{\dK}$
and the terms $\resDFKi{\wwh}{\zzh}$ and $\resDFIi{\wwh}{\zzh}$ are the the $i$-th components, 
$i=1,\ldots,4,$ of the local residual terms given
\eqref{eul:dualResiduals1} and \eqref{eul:dualResiduals2}. 
Altogether, we obtain 
\begin{align} 
\label{eul:etaII}
|\etaI(\wwh,\zzh) | \leq \etaII(\wwh,\zzh), \qquad
\etaII(\wwh,\zzh) = \sum_{K \in \mesh} \etaIIK, 
\end{align}
where 
\begin{align} 
 \label{eul:etaIIK}
 \etaIIK 
 =  \frac{1}{2}   \Big( \sum\nolimits_{i=1}^4&\, \RKVi\norm{(\zhp - \Pi \zhp)^i}{K} + \RKBi \norm{(\zhp - \Pi \zhp)^i}{\dK} 
   \\ 
    & +  \RDKVi\norm{(\whp - \Pi \whp)^i}{K} + \RDKBi\norm{(\whp - \Pi \whp)^i}{\dK} \Big) . \notag
\end{align}
The terms including $\zhp - \Pi \zhp$ and $\whp - \Pi \whp$ are called {\em weights} 
and the estimates shaped like \eqref{eul:etaII} are usually referred as 
{\em dual weighted residual} error estimate, cf. \cite{RannacherBook}.

\subsection{Goal-oriented anisotropic error estimates}
\label{sec:AMAhp}

The form of the local error estimate $\etaIIK$ given by \eqref{eul:etaIIK} can be directly used
for the recently proposed goal-oriented anisotropic $hp$-mesh adaptation in
\cite{DWR_AMA,ESCO-18} for scalar linear convection-diffusion problems.
In \cite{DWR_AMA}, we derived goal-oriented error estimates including
the anisotropy of mesh elements and proposed $h$-adaptive mesh algorithm.
Further, in \cite{ESCO-18}, we extended this technique to the $hp$-variant.
These estimates can be written in an abstract way as
\begin{align}
  \label{eul:etaIIIK}
  \etaIIK \leq \etaIIIK :=
  \GG\left( \RKV, \RKB, \RDKV, \RDKB, \up,\zhp ; \nu_K, \sigma_K, \phi_K, p_K \right),
\end{align}
where $\GG$ is a function,
$\RKV$, $\RKB$, $\RDKV$, $\RDKB$ are the primal and adjoint residuals analogous to
those ones in \eqref{resP} and \eqref{resD}. Moreover, $\up$ and $\zhp$ are the reconstructed
higher-order approximations of the primal and adjoint solutions, respectively.
Moreover, parameters  $\nu_K$, $\sigma_K$ and $\phi_K$ denote the size (area),
aspect ratio and orientation of an anisotropic element $K\in\Th$, respectively,
and $p_K$ is the corresponding polynomial approximation degree.
The explicit dependence of $\GG$ on the shape of a grid triangle)
(parameters $\nu_K$, $\sigma_K$, $\phi_K$) is the key results for the anisotropic mesh adaptation.
Using the estimate \eqref{eul:etaIIIK}, we developed an algorithm which 
locally optimizes $\sigma_K$ and $\phi_K$ while $\nu_K$ and $p_K$ are fixed such that
the density of degrees of freedom is constant. Hence, for each $K\in\Th$ we have to minimize
a functional of two variables.

The extension of this $hp$-mesh adaptation algorithm to the system of
the compressible Euler equations is relatively straightforward. Only difference is that
the right-hand side of \eqref{eul:etaIIIK} contains the residuals for all component
of $\w$ and $\z$, i.e., 16 terms altogether. Hence, the form of $\GG$ is more complicated,
however the minimizing algorithm from \cite{DWR_AMA,ESCO-18} works with a minor modification.
The whole adaptive computational process is written in Algorithm~\ref{alg:AMAhp}.
All technical details are in \cite{DWR_AMA,ESCO-18}.
In comparison to recently published anisotropic $hp$-mesh adaptation algorithm
in \cite{RanMayDol-JCP20}, estimate \eqref{eul:etaIIIK}  does not employ the Lipschitz
continuity of convective fluxes.

\begin{algorithm}[t]
  \caption{Goal-oriented anisotropic $hp$-mesh adaptive algorithm}
  \label{alg:AMAhp}
  \begin{algorithmic}[1]
    \STATE let $\Thz$ be the initial (coarse) mesh and $\bShzp$ be the
    corresponding DG space 
    \FOR{$m=0,1,\dots$} 
    \STATE solve problems \eqref{eul:discPrimal} and 
    \eqref{eul:discrAdjoint} with outputs $\whm,\zzhm\in\bShmp$
    \STATE set reconstructions $\whmp,\zzhmp\in\bShmpp$ using \eqref{eul:rec} 
    \STATE  evaluate
    $\etaI(\whm,\zzhm)$ and $\etaIK,\ K\in\Thm$ using \eqref{eul:residuals} and \eqref{eul:etaIK}
    \IF {$\etaI\leq \TOL$} \STATE STOP the computation
    \ELSE
    \STATE using $\etaIK$, propose a new size of  $K\in\Thm$
    \STATE using $\etaIIIK$, optimize the shape of $K$ and polynomial degree $p_K$ for $K\in\Thm$
    \STATE generate new mesh $\ThmP$ and DG space $\bShmP$
    \ENDIF
    \ENDFOR
    
  \end{algorithmic}
\end{algorithm}

\section{Numerical experiments} 
\label{sec:numer}
In this section we present several experiments which support the theoretical
results presented in previous sections. First we show that the adjoint
problem \eqref{eul:discrAdjoint} produces a smooth adjoint solution which justifies the
adjoint consistency. Furthermore, we demonstrate the performance of the
anisotropic $hp$-mesh adaptation Algorithm~\ref{alg:AMAhp},
namely the convergence of the error $J(\w) - J(\wh)$ and its estimates
$\etaI$, $\etaII$ with respect to the number of degrees of freedom ($\DoF=\dim\bShp$). 
We consider several subsonic and transonic flows around  NACA0012 profile characterized
by the far-field Mach number $\Min$ and the  angle of attack  $\alpha$.

\subsection{Adjoint consistency of the DG discretization} 

The goal of this section is show that the treatments of the impermeable boundary
conditions on $\gomW$ by $\HboundWone$ and $\HboundWtwo$ from  \eqref{eul:BCw:boundaryOperator} and
\eqref{eul:mirBC}, respectively, together with the consistent modification of $\Jd$
by \eqref{eul:Jd} produce a smooth approximate adjoint solution $\zzh$ which supports
the adjoint consistency \eqref{eul:adjointConsistency} proved in Theorem~\ref{the:dual_cons}.

We consider the flow with the inlet Mach number $\Min = 0.5$ 
and the  angle of attack  $\alpha = 0 \degree$. The quantity of interest is the drag
coefficient defined by \eqref{eul:cDL1}.
We employ a fixed triangular mesh, initially refined in the vicinity of the profile, 
which is shown in Figure~\ref{fig:euler_consistency_mesh} together with the
isolines of density obtained by $P_3$ polynomial approximation.

Figure \ref{fig:euler_consistency} compares the isolines of all components of the solution 
of the discrete adjoint problem \eqref{eul:discrAdjoint} using the numerical flux 
$\HboundWtwo$ given by \eqref{eul:mirBC} 
accompanied with the modified target functional \eqref{eul:Jd}  
and with the original target functional \eqref{eul:cDL} employing the proper
differentiation given by \eqref{eul:DJ}.

We see that while the adjoint consistent discretization leads to a smooth solution $\zzh$, 
the inconsistent one contains non-physical oscillations.
These results justify that the discrete adjoint problem \eqref{eul:discrAdjoint} is well-posed.
Furthermore, the discretization using the numerical flux 
$\HboundWone$ given by \eqref{eul:BCw:boundaryOperator}
leads to very similar results so we do not show them.
Finally, we note that our results are in agreement with 
\cite[Section 6.1]{Hartmann2015Generalized} 
where a similar example is presented, but with proper differentiation of the form $\bbh.$

%

\begin{figure}
  \begin{center} 
    \includegraphics[height=0.36\textwidth]{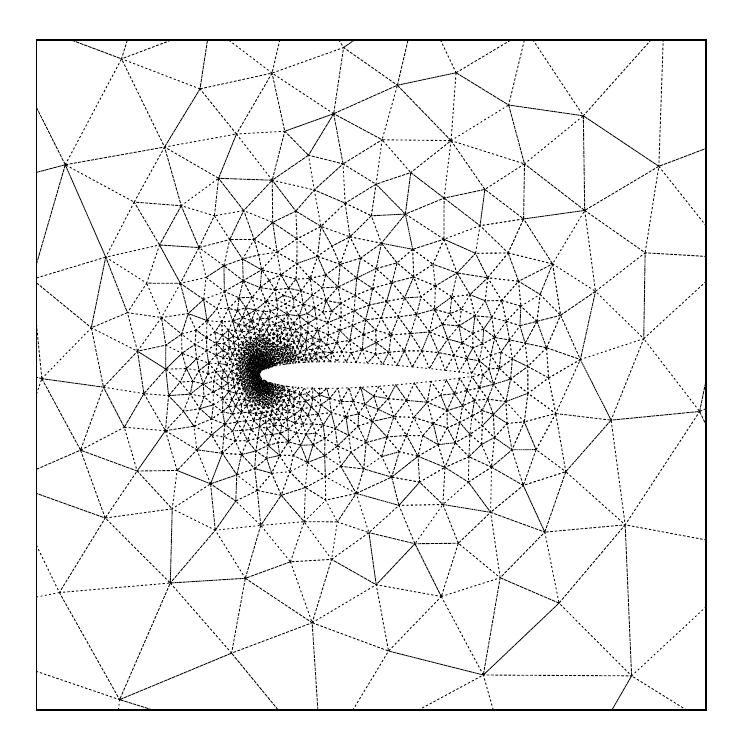}
    \includegraphics[height=0.35\textwidth]{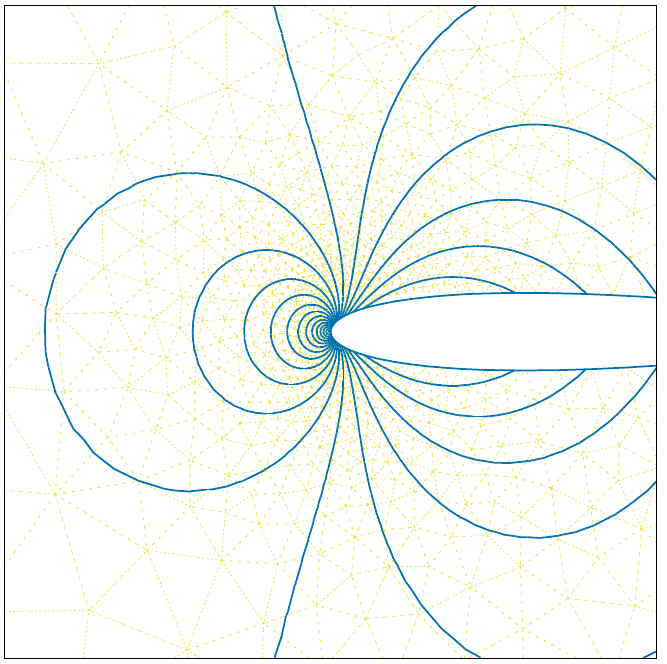}
    \caption{Subsonic inviscid flow around the NACA 0012 profile
      ($\Min = 0.5$, $\alpha = 0\degree$): 
      the computational mesh in the vicinity of the profile (left) and 
      the isolines of the first component of $\wwh$ (right).}
    \label{fig:euler_consistency_mesh}
  \end{center}
\end{figure}

\begin{figure}
  \begin{center}
    $\zzhi{1}$
    \hspace{25mm} $\zzhi{2}$
    \hspace{25mm} $\zzhi{3}$
    \hspace{25mm} $\zzhi{4}$
    
    \includegraphics[height=0.24\textwidth]{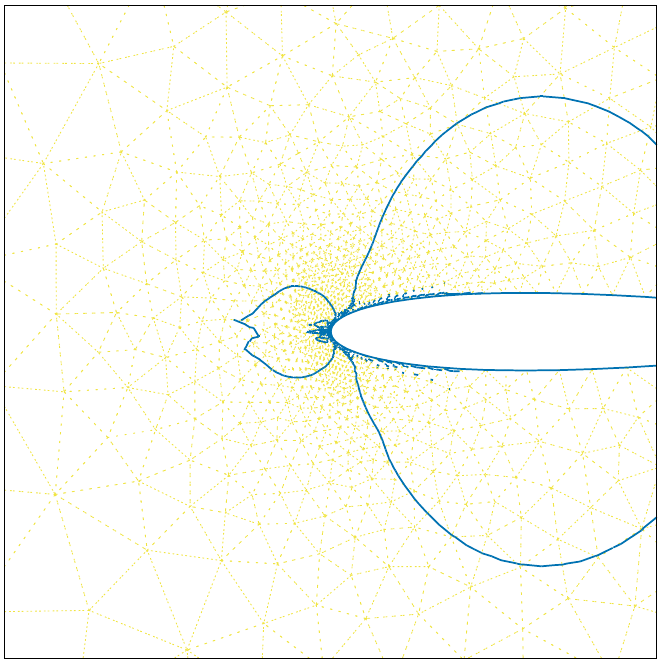}
    \includegraphics[height=0.24\textwidth]{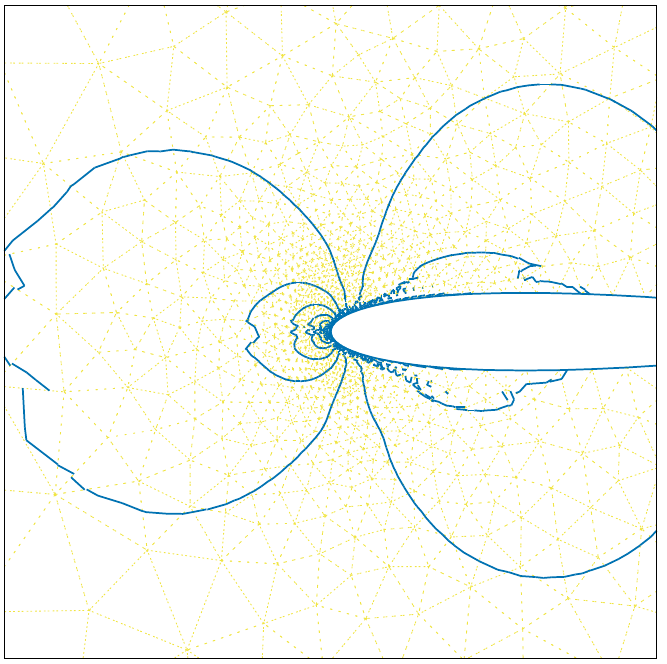}
    \includegraphics[height=0.24\textwidth]{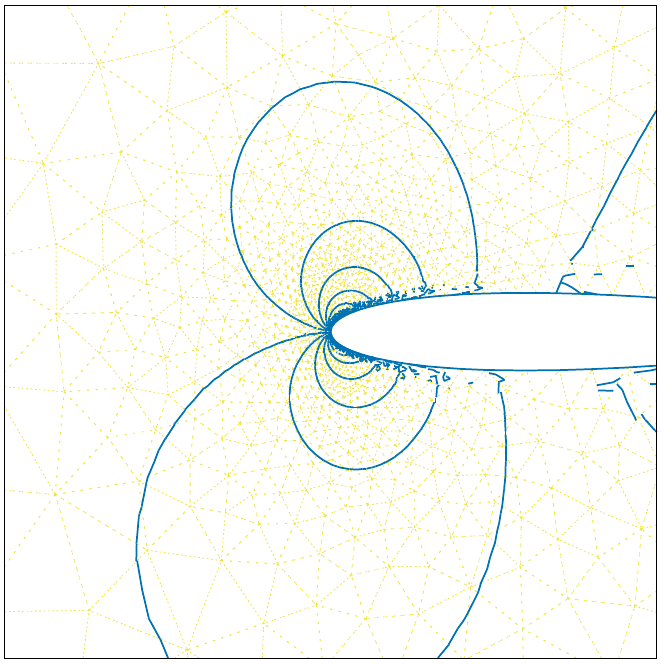}
    \includegraphics[height=0.24\textwidth]{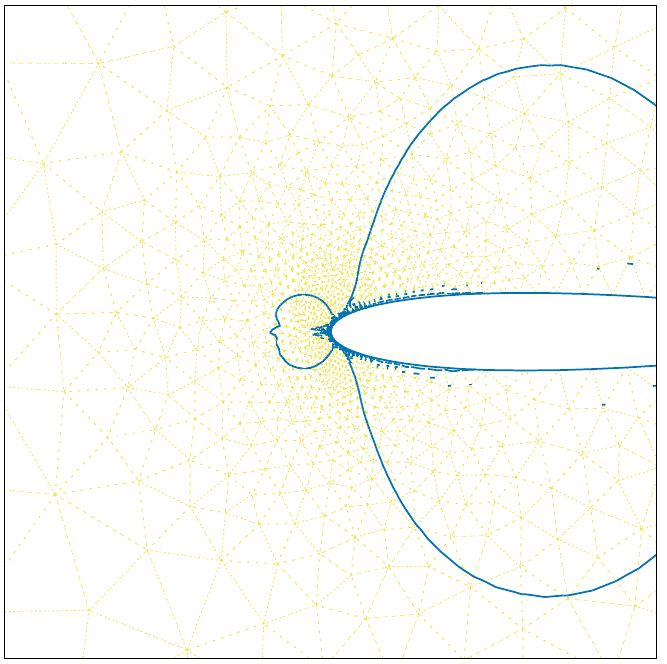}
    
    \includegraphics[height=0.24\textwidth]{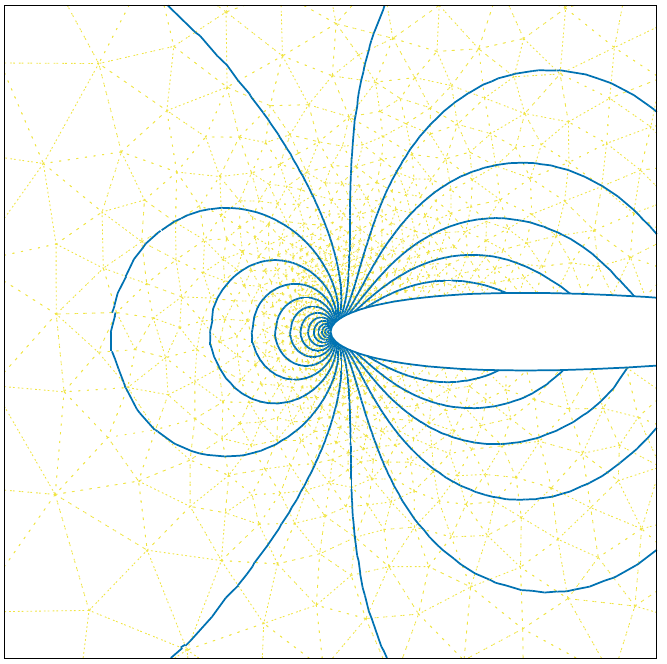}
    \includegraphics[height=0.24\textwidth]{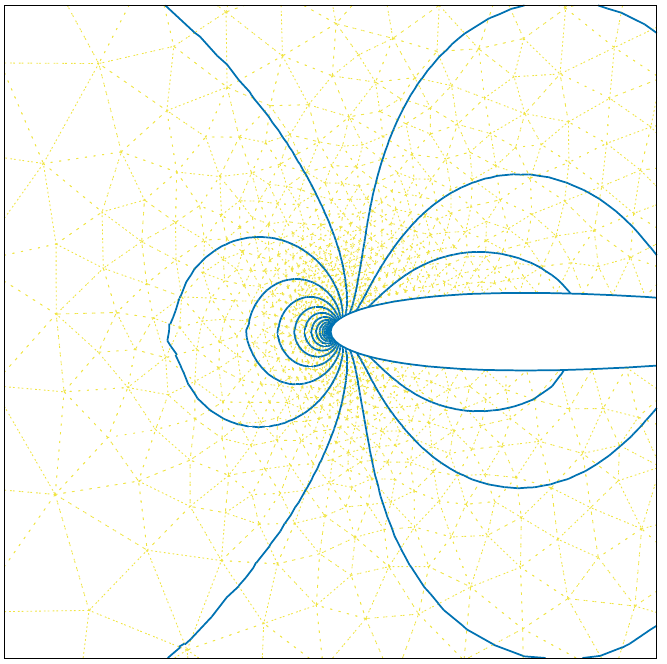}
    \includegraphics[height=0.24\textwidth]{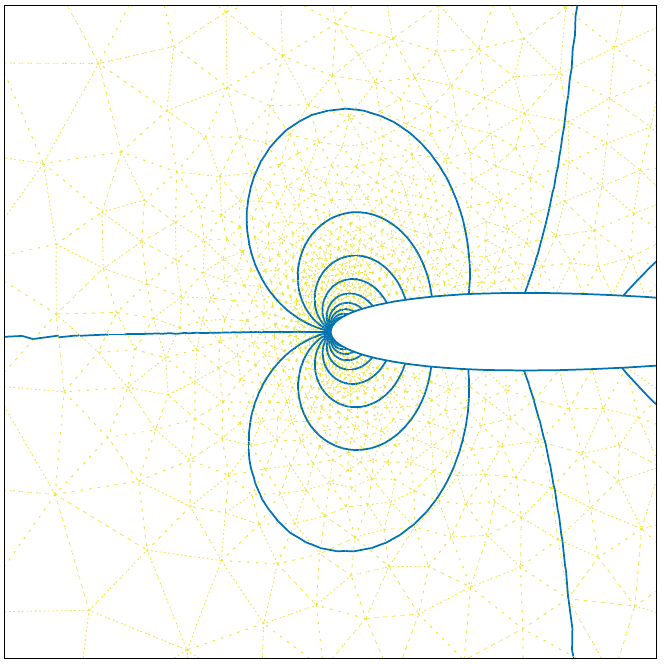}
    \includegraphics[height=0.24\textwidth]{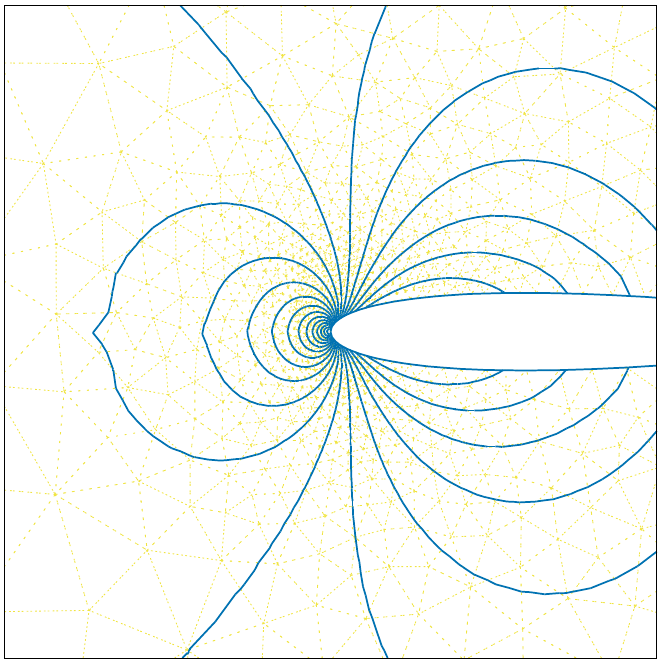}
    
    \caption{Subsonic inviscid flow around the NACA 0012 profile
      ($\Min = 0.5$, $\alpha = 0\degree$), $J=$ drag coeff.: 
      inconsistent (top) and consistent (bottom) discretization,
      the components of the discrete adjoint solutions
      $\zzh = (\zzhi{1}, \zzhi{2}, \zzhi{3}, \zzhi{4})$ are shown.
    }
    \label{fig:euler_consistency}
  \end{center}
\end{figure}
 
\subsection{Subsonic flows}
In the following examples, we demonstrate the performance of the
anisotropic mesh adaptive Algorithm~\ref{alg:AMAhp}, from Section~\ref{sec:AMAhp},
namely its $h$- and $hp$-variants.

\subsubsection{Symmetric subsonic flow}
We consider again the flow with $\Min =0.5$, $\alpha =0\degree$ and the target functional
$\J(\ww)$ is  the drag coefficient according to \eqref{eul:cDL1}.
The exact value of the drag coefficient is $c_D = 0$. 
Figure~\ref{fig:drag:decrease} shows the decrease of the error of the target quantity
$\Jd(\wh) - \J(\ww)= \Jd(\wh)$ and the corresponding estimates
$\etaI$ and $\etaII$ (cf. \eqref{eul:etaI} and \eqref{eul:etaII})
\wrt $\DoF$
for the $h$- (with $p=2$) and $hp$-version of the mesh adaptive algorithm.
We observe that both $\etaI$ and $\etaII$ approximate the true error quite 
accurately although $\etaI$ underestimates the error slightly. 
We see that the $hp$-version is superior to the $h$-version, as expected.
Moreover, for the $hp$-variant, the error starts to stagnate at the level approximately
5E-07. We discuss this effect in Section~\ref{sec:subL}, where it is better to observe.


Figure \ref{fig:dragHPmeshes} shows the details of the $hp$-meshes of the anisotropic
$hp$-adaptation after the 5th and 13th (the last) levels of adaptation. 
There is a strong $h$-refinement in the vicinity of the trailing edge due to the singularity,
and outside of this small region, the strong $p$-adaptation is performed since the solution
is sufficiently smooth.

\begin{figure}
\begin{center}
  \includegraphics[width=0.47\textwidth]{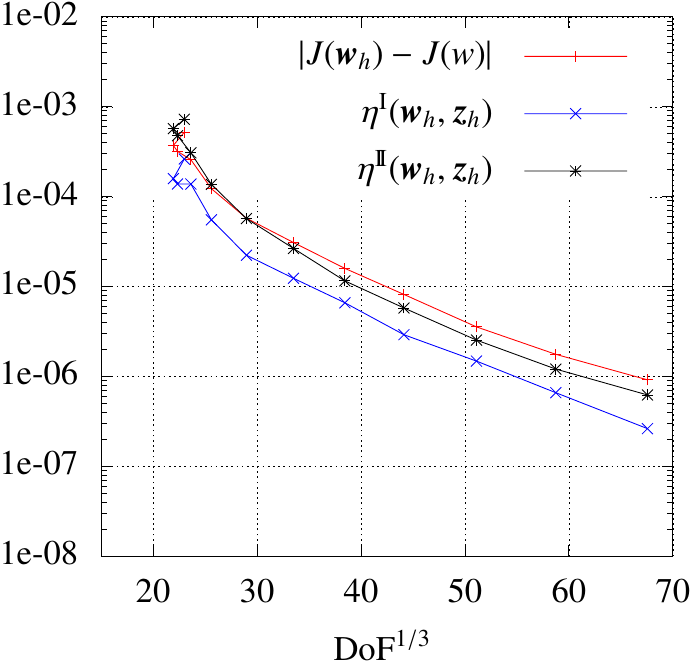}
  \hspace{0.03\textwidth}
  \includegraphics[width=0.47\textwidth]{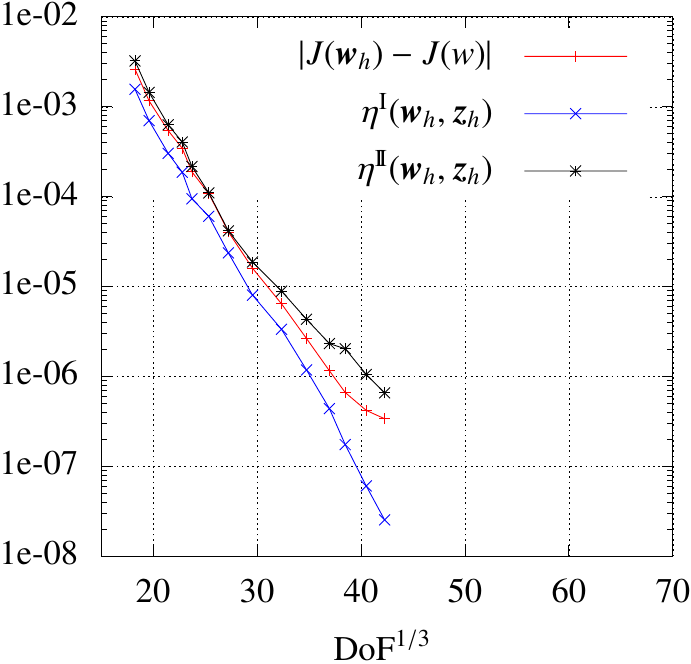}
  \caption{Subsonic inviscid flow around the NACA 0012 profile
    ($\Min = 0.5$, $\alpha = 0\degree$), $J=$ drag coeff.: 
   decrease of the error $J(\w) - J(\wwh)$ and the goal-oriented error estimates $\etaI$ and $\etaII$ 
   with respect to the cube root of DOF 
   for the $h$-refinement using $p=2$ DG approximations (left)
   and the $hp$-version (right).
   }
   \label{fig:drag:decrease}
\end{center}
\end{figure}

\begin{figure}
  \begin{center}

    \includegraphics[width=0.48\textwidth]{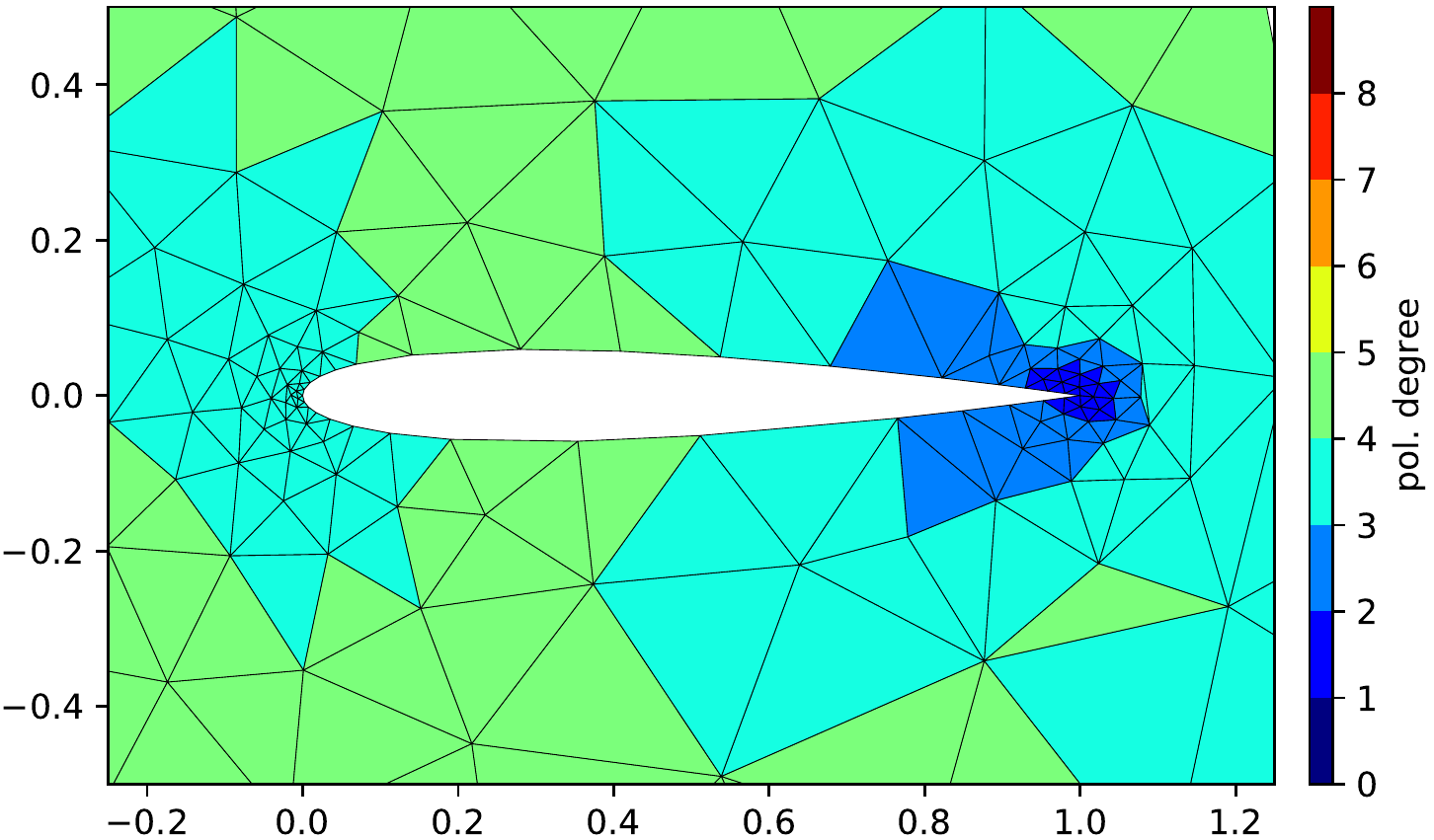}
    \hspace{1mm}
    \includegraphics[width=0.48\textwidth]{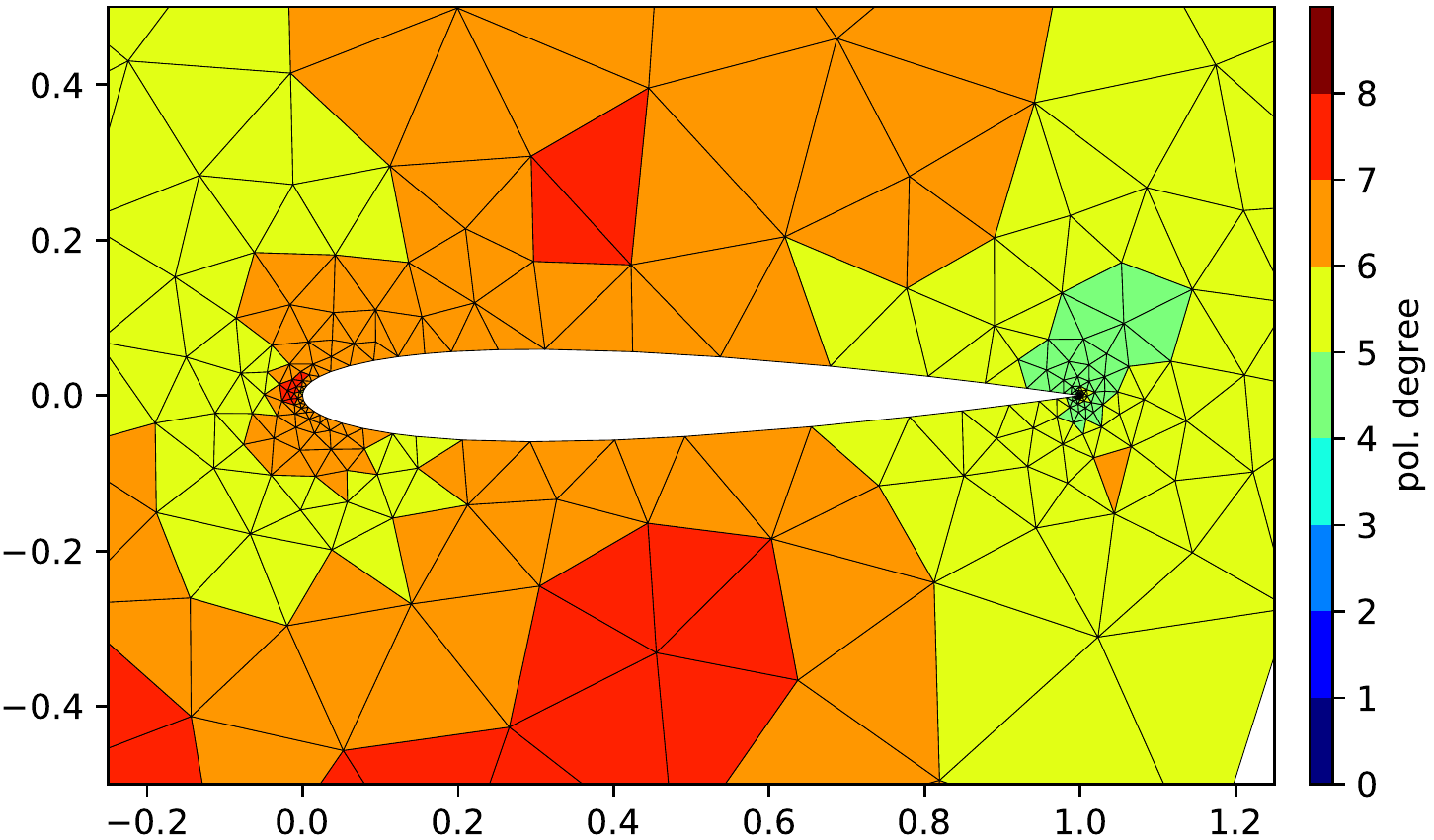}
    
    \includegraphics[width=0.48\textwidth]{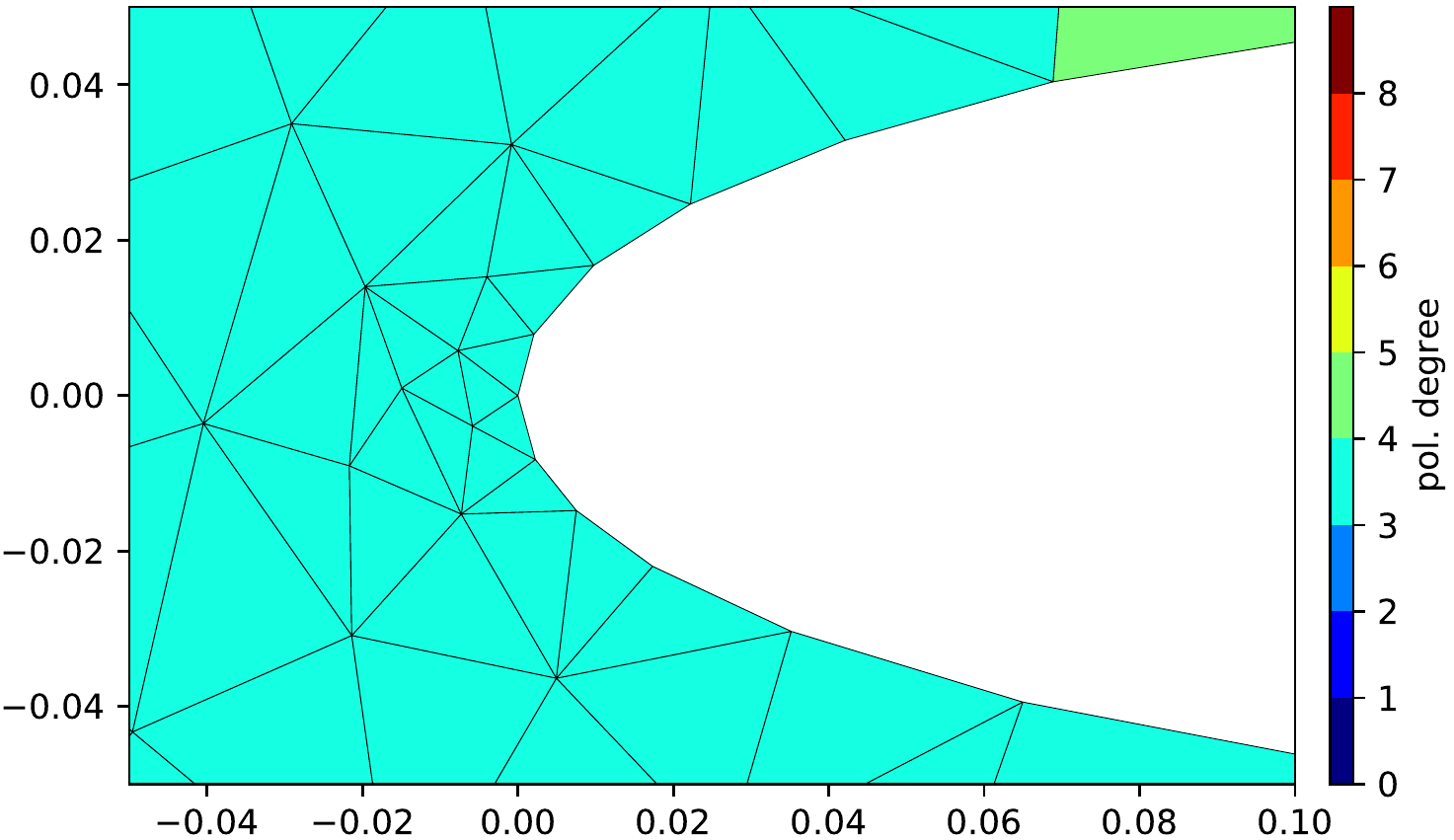}
    \hspace{1mm}
    \includegraphics[width=0.48\textwidth]{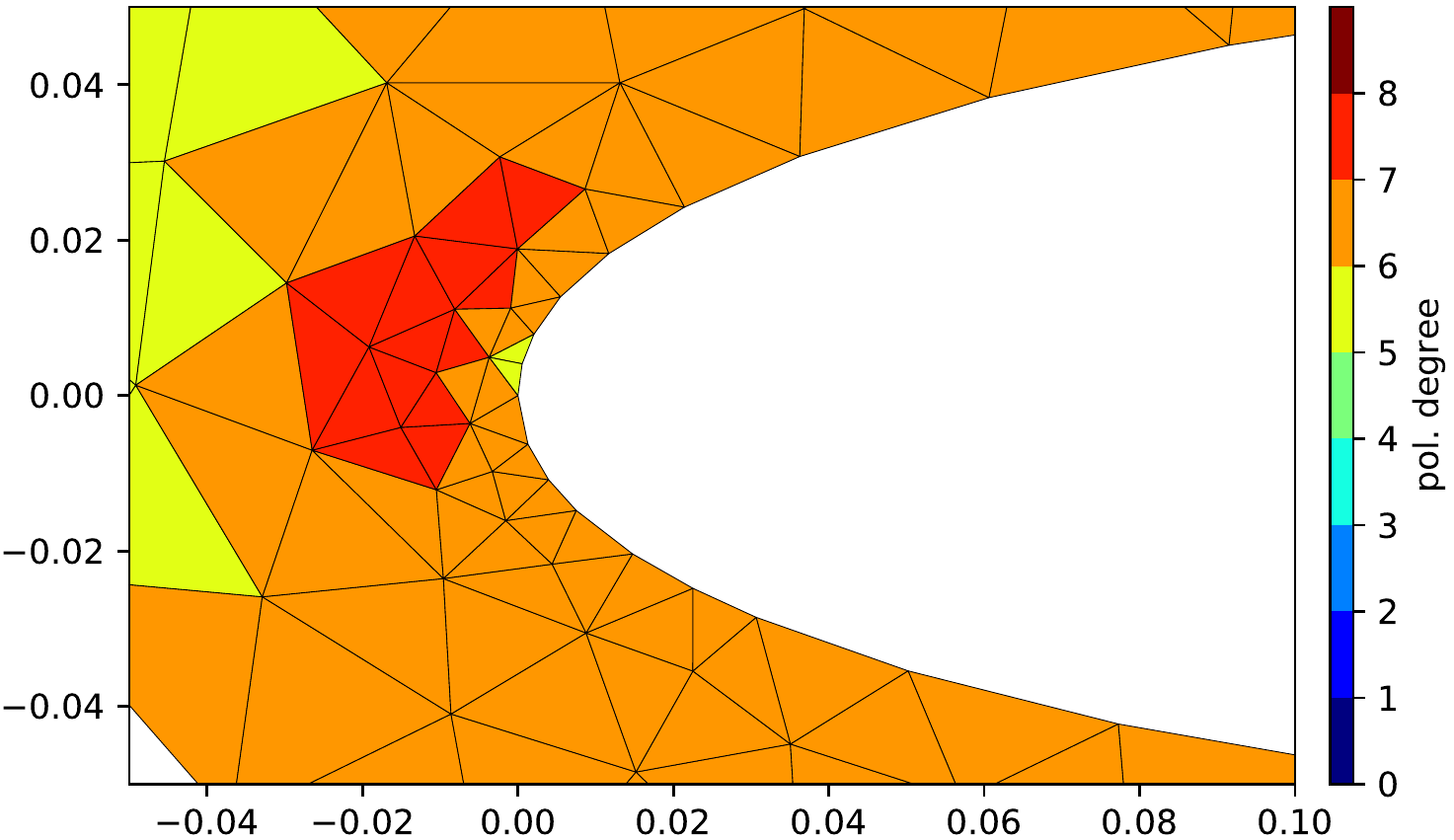}
    
    \includegraphics[width=0.48\textwidth]{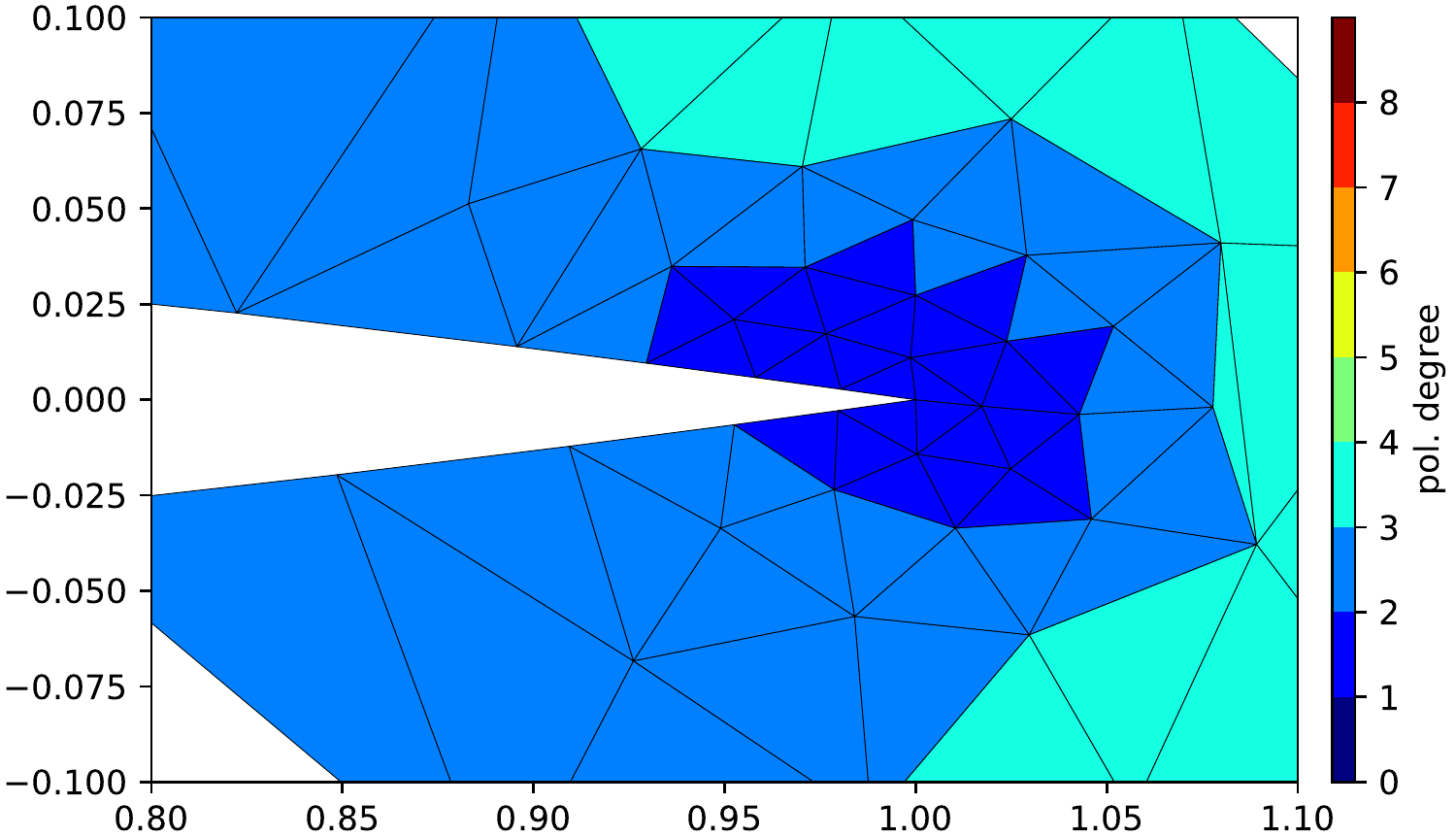}
    \hspace{1mm}
    \includegraphics[width=0.48\textwidth]{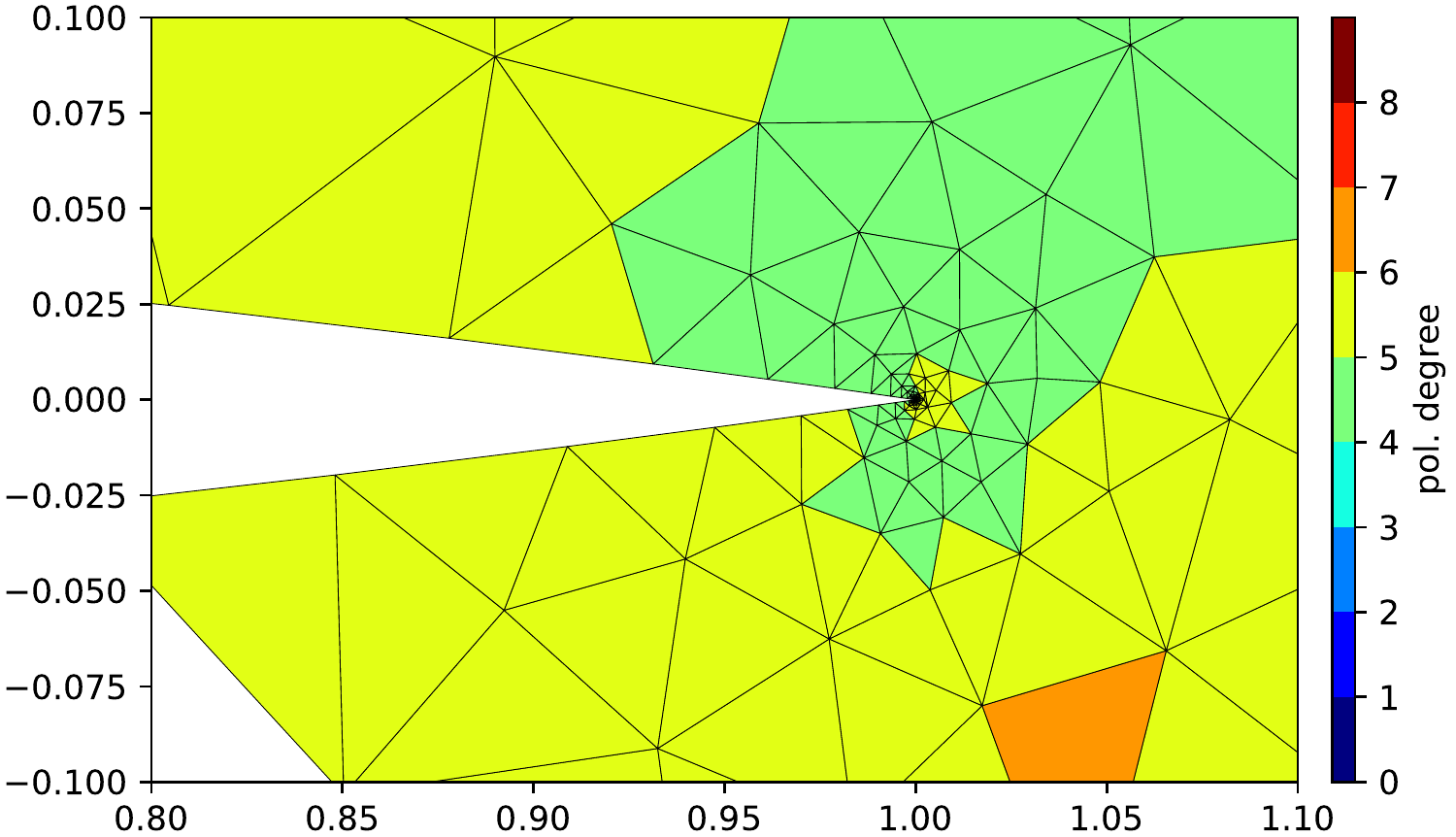}
    
    \caption{Subsonic inviscid flow around the NACA 0012 profile
      ($\Min = 0.5$, $\alpha = 0\degree$), $J=$ drag coeff.: 
      local polynomial degrees after the 5th (left) and 13th (right) levels of 
      the anisotropic $hp$-mesh adaptation, the whole profile (top) 
      and zooms of the leading (middle) and trailing (bottom) edge of the profile.
    }
   \label{fig:dragHPmeshes}
   \end{center}
\end{figure}

\subsubsection{Non-symmetric subsonic flow}
\label{sec:subL}
We consider the flow with $\Min =0.5$ and $\alpha =1.25\degree$
and the target functional is the drag as well as lift coefficients.
Whereas the exact value of $c_D$ is again zero,
the exact value of the lift coefficient has to be computed experimentally.
We use the reference value of $c_L^{\mathrm{ref}}= 1.757 \cdot 10^{-1} \pm 10^{-4}$
achieved by the $hp$-adaptive algorithm.

Figure~\ref{fig:suaD} shows the decrease of the error of the drag 
coefficient and their estimates \wrt $\DoF$.
We observe that whereas both estimates are decreasing, the exact error stagnates at the
level slightly below 1E-4. It means that the drag coefficient does not converge to
the exact value $J(\w)=0$
but to a positive value $c_D^*$. This effect was investigated in details in
\cite{VassbergJameson_JA10} using several codes with a strong global refinement.
Each of the tested code gave a small positive limit value $c_D^*$ (obtained by the Richardson
extrapolation). Figure~\ref{fig:suaD} shows also the quantity
$|J(\wh)-c_D^*|$ with $c_D^*= 6.8\cdot 10^{-5}$ which already converges as expected.

Furthermore, Figure~\ref{fig:suaL} shows the decrease of the error of the lift 
coefficient and their estimates \wrt $\DoF$.
We see that the error estimates work worse than in the previous case
-- $\etaI$ underestimates the error 
and, quite the other way, $\etaII$ overestimates it almost ten times. 
This may be caused by the weaker regularity of the adjoint solution for the lift coefficient. 
In order to support this conjecture, we present Figure~\ref{fig:dragVSlift}, which
compares the first component of the adjoint solutions $\zzh$ and the
corresponding $hp$-meshes for the drag coefficient  and lift coefficients.
These results indicate that
\begin{compactitem}
\item[(i)] the adjoint problem 
  for the drag coefficient is quite smooth and then $p_K$ are
  high for $K$ in the surrounding of the profile $\gomW$,
\item [(ii)] the adjoint problem for the lift coefficient has less regularity due to
  ``boundary layers'' along the profile
  and consequently, the strong $h$-refinement with low polynomial degree $p_K$ is presented
  for $K$ close to $\gomW$.
\end{compactitem}
We remark that in the majority of articles on goal-oriented error estimates 
for the Euler equations, e.g.,
\cite{Hartmann2005Role,Hartmann2006Derivation,Hartmann2007Adjoint,Hartmann2015Generalized}, 
the numerical experiments are performed only for the drag coefficient.
We have found only one experiment 
with the lift coefficient in \cite{Sharbatdar2018Mesh} 
for the transonic flow around the NACA 0012 profile.

\begin{figure}
  \begin{center}
  \includegraphics[width=0.45\textwidth]{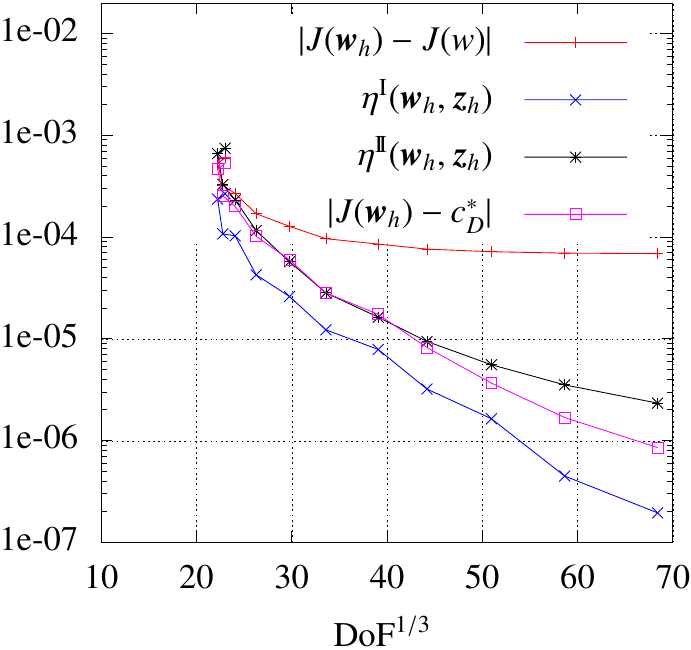}
  \hspace{0.02\textwidth}
  \includegraphics[width=0.45\textwidth]{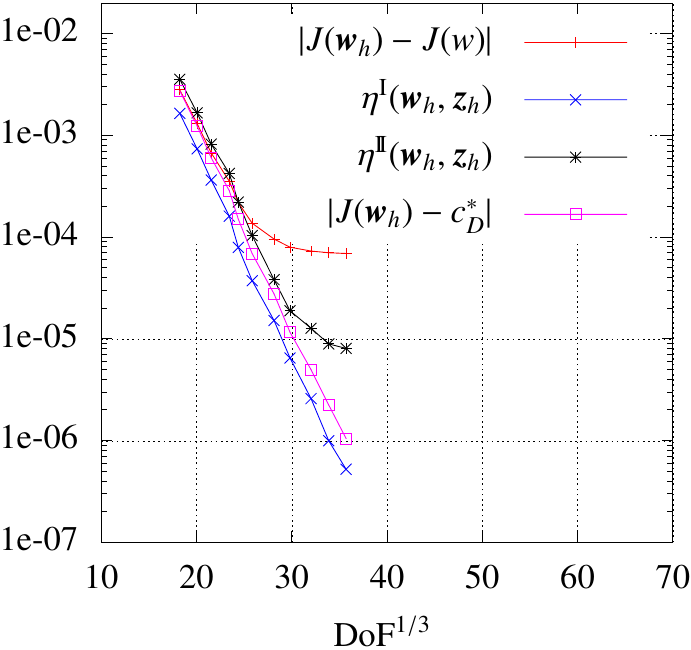}

   \caption{Subsonic inviscid flow around the NACA 0012 profile
     ($\Min = 0.5$, $\alpha = 1.25\degree$), $J=$ drag coeff.:
     decrease of the error $J(\w) - J(\wwh)$ and the goal-oriented error estimates
     $\etaI$ and $\etaII$ 
     with respect to the cube root of $\DoF$ 
     the $h$-refinement using $p=2$ DG approximations (left)
     and the $hp$-version (right). }
   \label{fig:suaD}
   \end{center}
\end{figure}

\begin{figure}
  \begin{center}
  \includegraphics[width=0.45\textwidth]{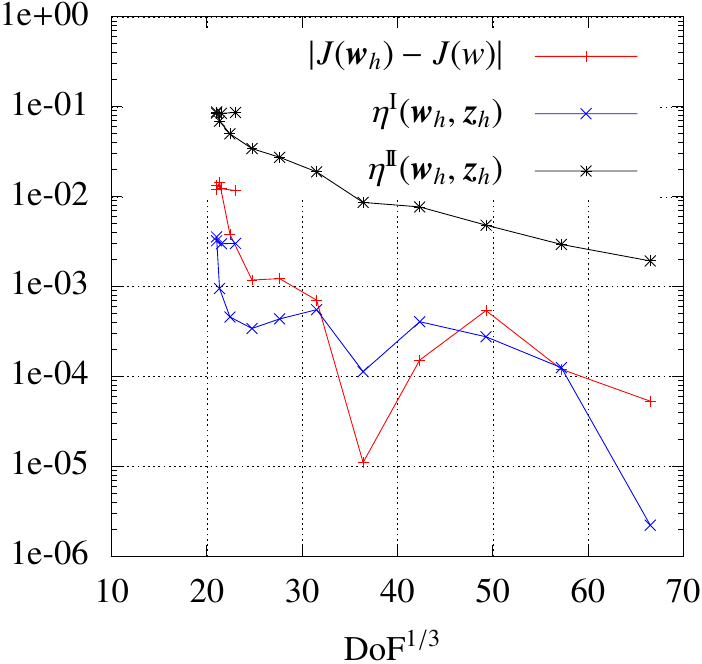}
  \hspace{0.02\textwidth}
  \includegraphics[width=0.45\textwidth]{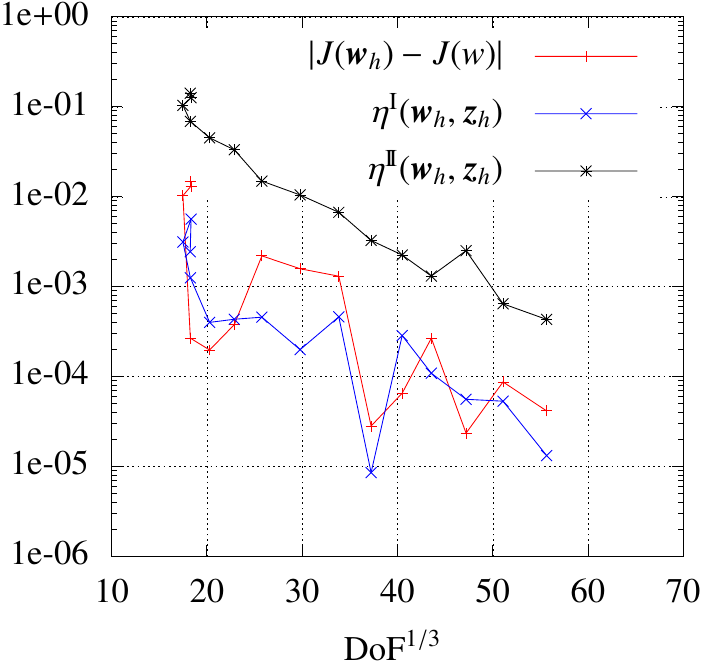}
   \caption{Subsonic inviscid flow around the NACA 0012 profile
     ($\Min = 0.5$, $\alpha = 1.25\degree$), $J=$ lift coeff.: 
     decrease of the error $J(\w) - J(\wwh)$ and the goal-oriented error estimates
     $\etaI$ and $\etaII$ 
     with respect to the cube root of $\DoF$ 
     the $h$-refinement using $p=2$ DG approximations (left)
     and the $hp$-version (right). }
   \label{fig:suaL}
   \end{center}
\end{figure}

\begin{figure}
  \begin{center}
    \vertical{\hspace{5mm} drag coefficient}
    \includegraphics[width=0.48\textwidth]{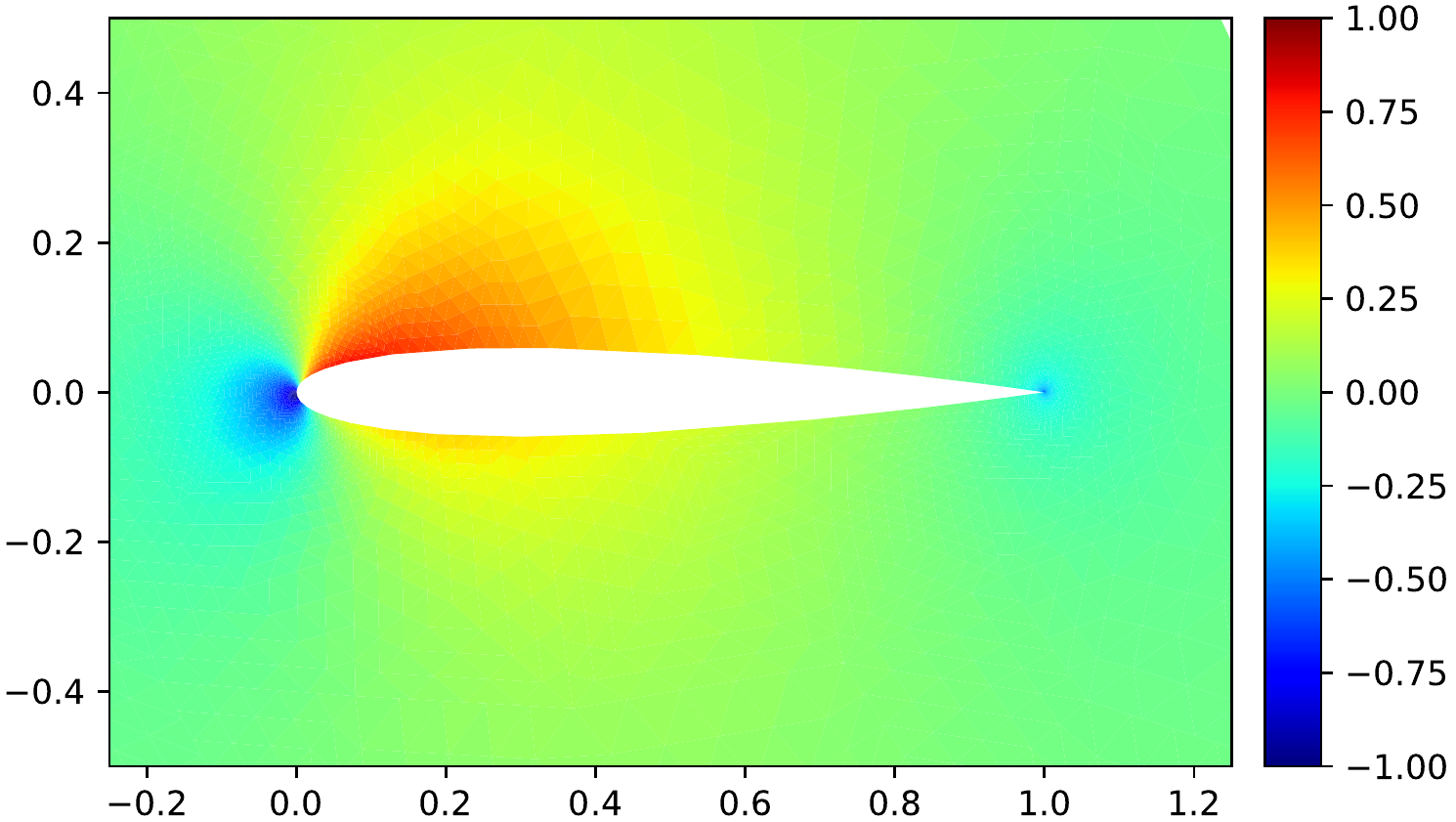}
    \hspace{0.02\textwidth}
    \includegraphics[width=0.45\textwidth]{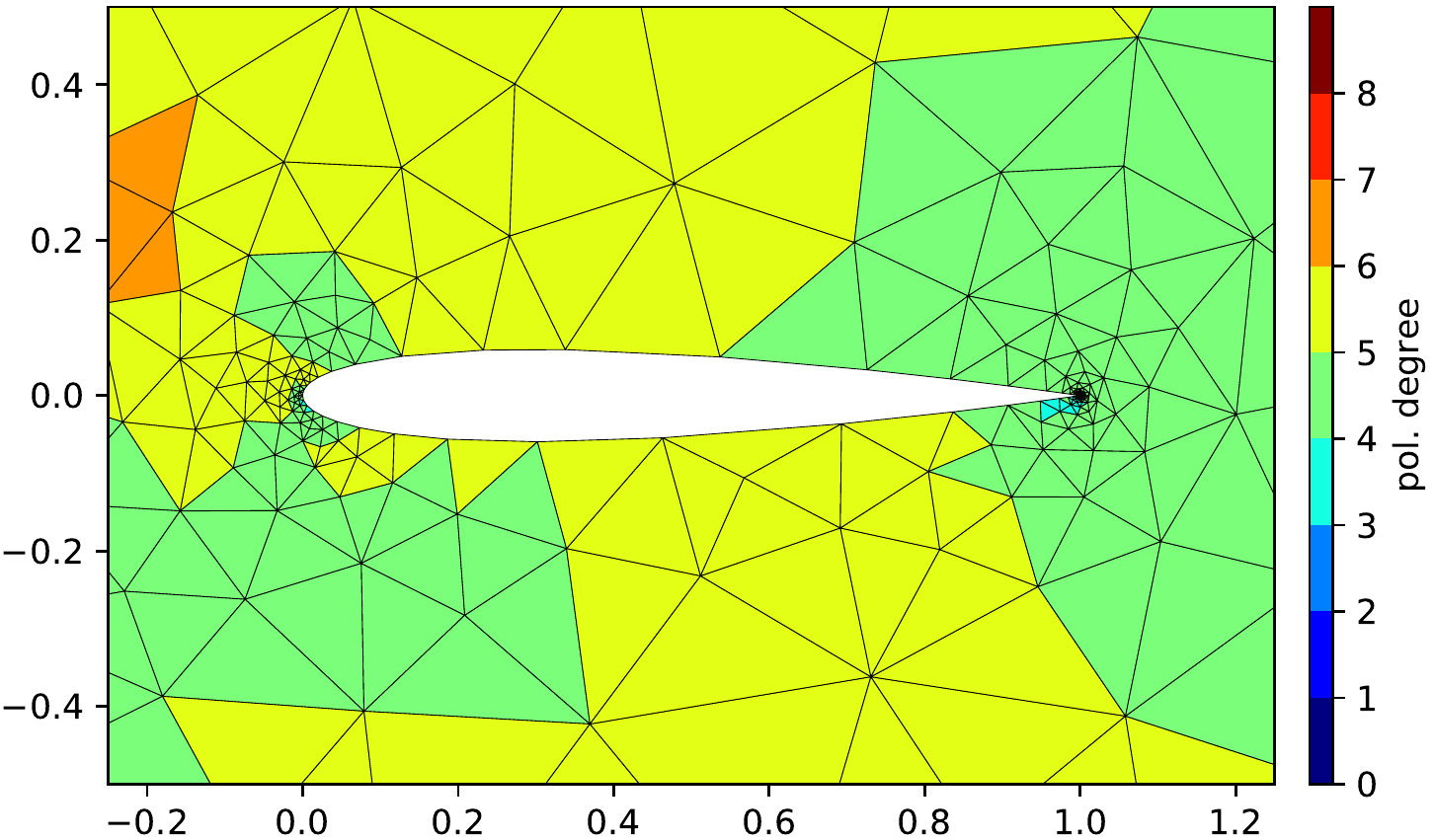}

    \vspace{3mm}
    
    \vertical{\hspace{5mm} lift coefficient}
    \includegraphics[width=0.48\textwidth]{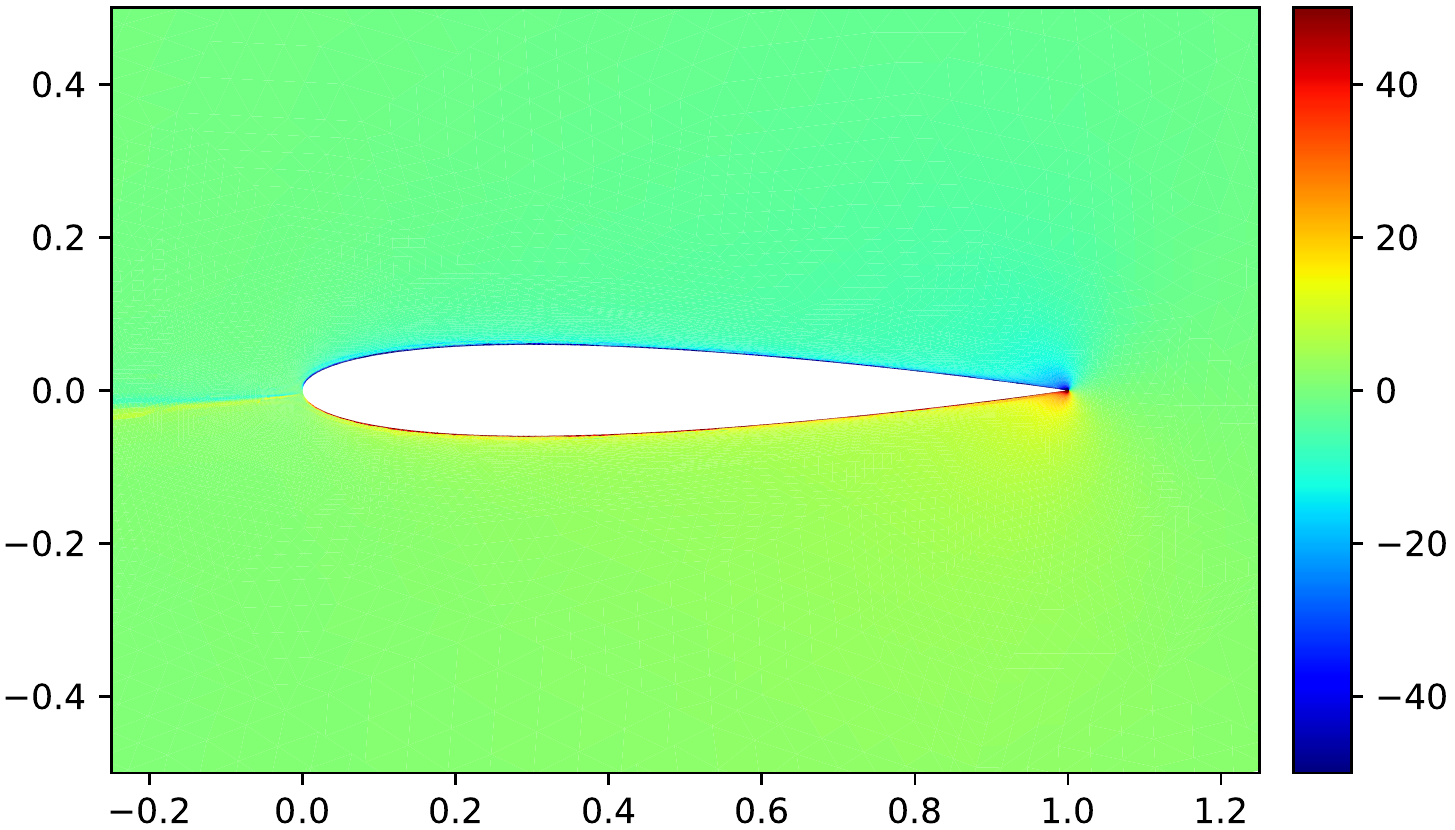}
    \hspace{0.02\textwidth}
    \includegraphics[width=0.45\textwidth]{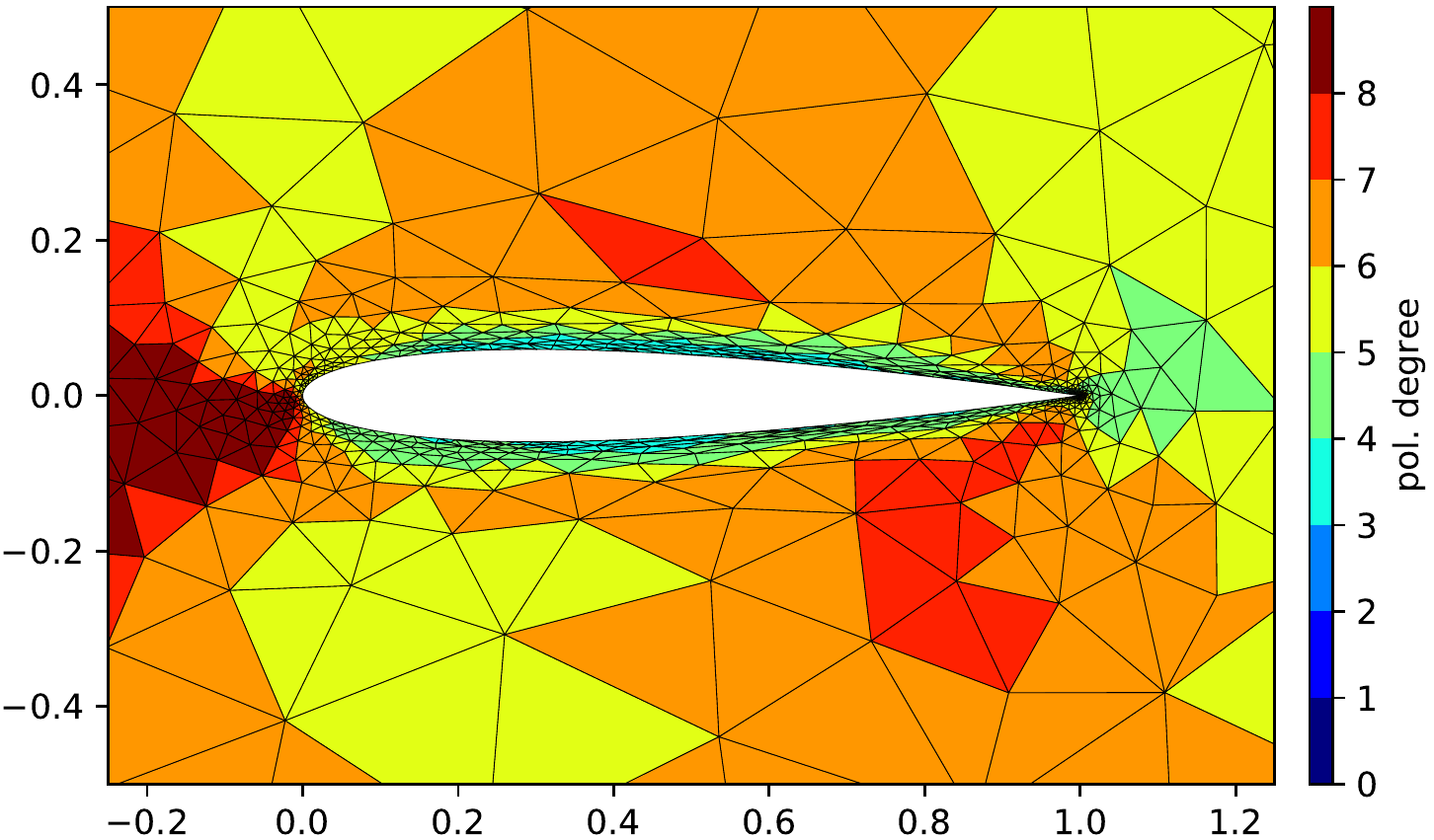}
    \caption{Subsonic inviscid flow around the NACA 0012 profile
      ($\Min = 0.5$, $\alpha = 1.25\degree$): 
      the first component of the discrete adjoint solution (left) on the final $hp-$mesh
      (right) for $J$ equal to the drag (top)  and lift (bottom) coefficients.}
   \label{fig:dragVSlift}
   \end{center}
\end{figure}

\subsection{Transonic flow} 
We consider the flow with  $\Min = 0.8$ and $\alpha= 1.25\degree$ which leads to two shock waves.
We apply the shock-capturing technique based on the artificial viscosity
whose amount is given by the jump indicator, see \cite[Section~8.5]{DGM-book} or
\cite{feikuc2007}.
The target functional $\J$ is both drag and lift coefficients,
the reference values $c_D^{\mathrm{ref}} = 2.135 \times 10^{-2}$
and  $c_L^{\mathrm{ref}} = 3.33 \times 10^{-1}$ were computed by the $hp$-anisotropic 
adaptation method. 


Figure~\ref{fig:trans-drag:decrease} shows the decrease of the error of the target quantity
and the error estimates $\etaI$ and $\etaII$  \wrt $\DoF$.
For both target functionals, $\etaI$ underestimates and $\etaII$ overestimates the true error by
a factor at most $10$. We suppose that such overestimation is caused by 
the high-order reconstruction from Section~\ref{sec:nonlReconstruct}, which 
is not sufficiently accurate for problems having discontinuous solution.
Moreover, Figure~\ref{fig:trans} shows the final $hp$-grids and the corresponding
distribution of the Mach number and the first component of the adjoint solution.
A strong $h$-refinement with low polynomial approximation degrees along
both shock waves is observed. A sharp capturing of both waves are easily to see.


\begin{figure}
\begin{center}
  \includegraphics[width=0.47\textwidth]{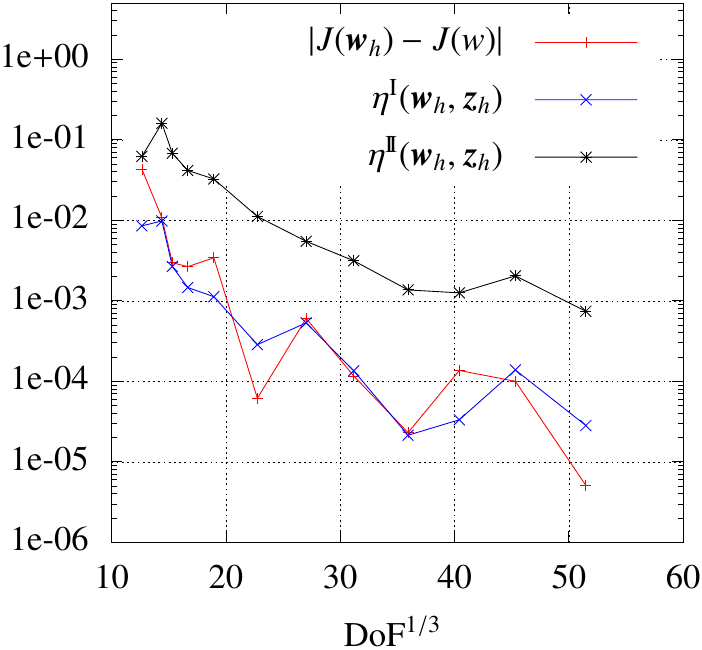}
  \hspace{0.03\textwidth}
  \includegraphics[width=0.47\textwidth]{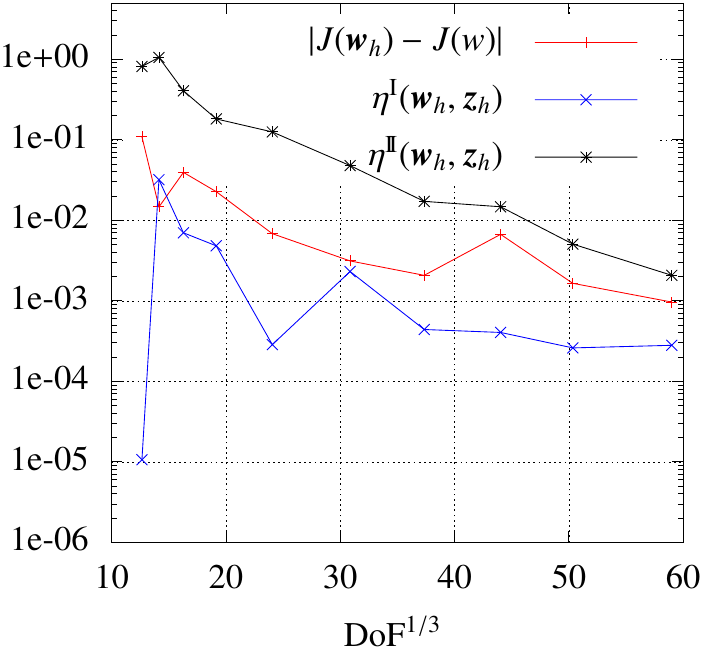}
     \caption{Transonic inviscid flow around the NACA 0012 profile
       ($\Min = 0.8$, $\alpha = 1.25\degree$): 
       decrease of the error $J(\w) - J(\wwh)$ and the goal-oriented error
       estimates $\etaI$ and $\etaII$  
       for $J$ equal to 
       the drag (left) and lift (right) coefficients
       w.r.t. the cube root of DOF.} 
   \label{fig:trans-drag:decrease}
   \end{center}
\end{figure}

\begin{figure}
  \begin{center}

    $J(u)$=drag coefficient \hspace{30mm} $J(u)$=lift coefficient \qquad \qquad

    \vertical{\hspace{12mm}$hp$-mesh}
  \includegraphics[width=0.46\textwidth]{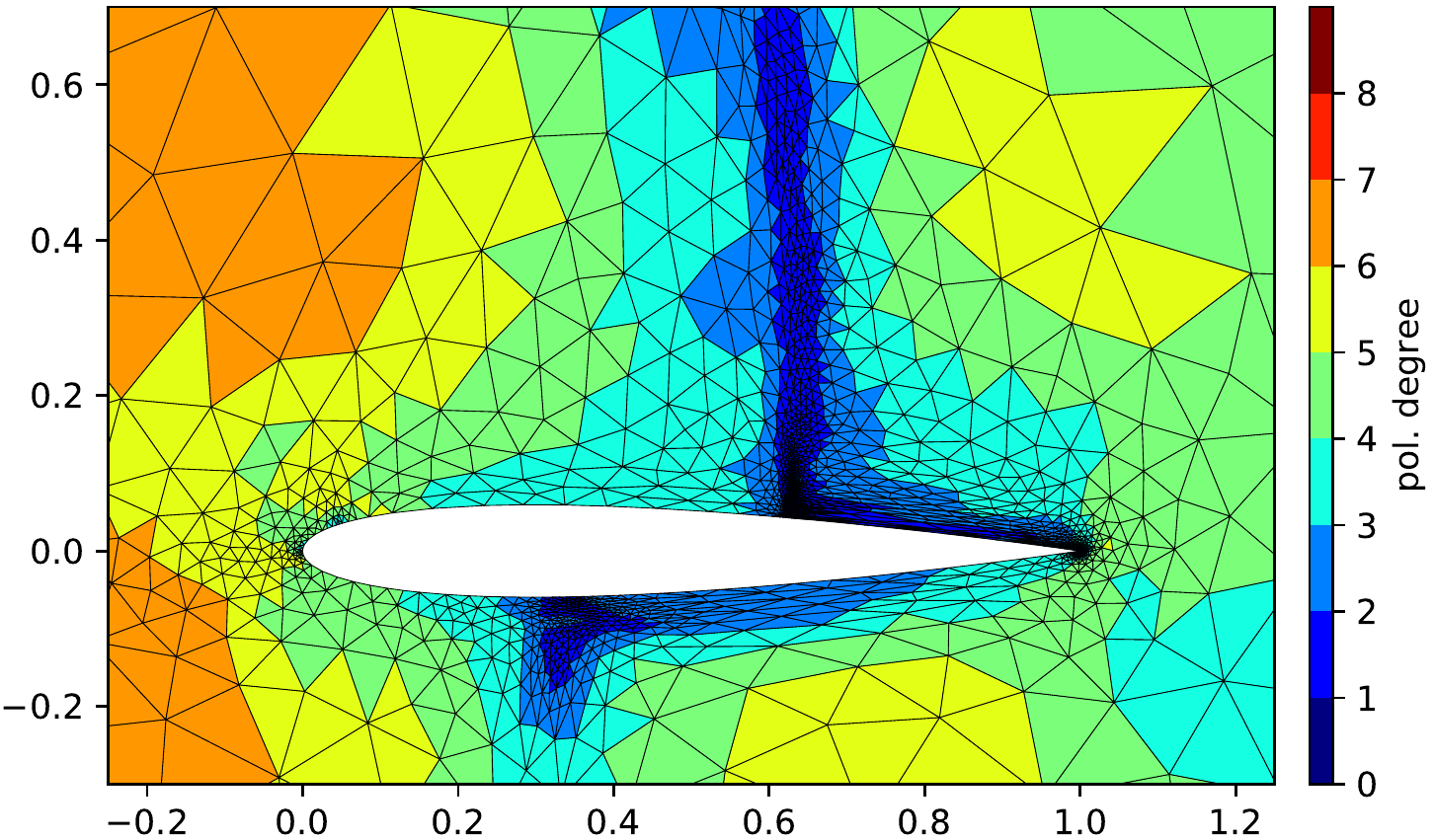}
  \hspace{1mm}
  \includegraphics[width=0.46\textwidth]{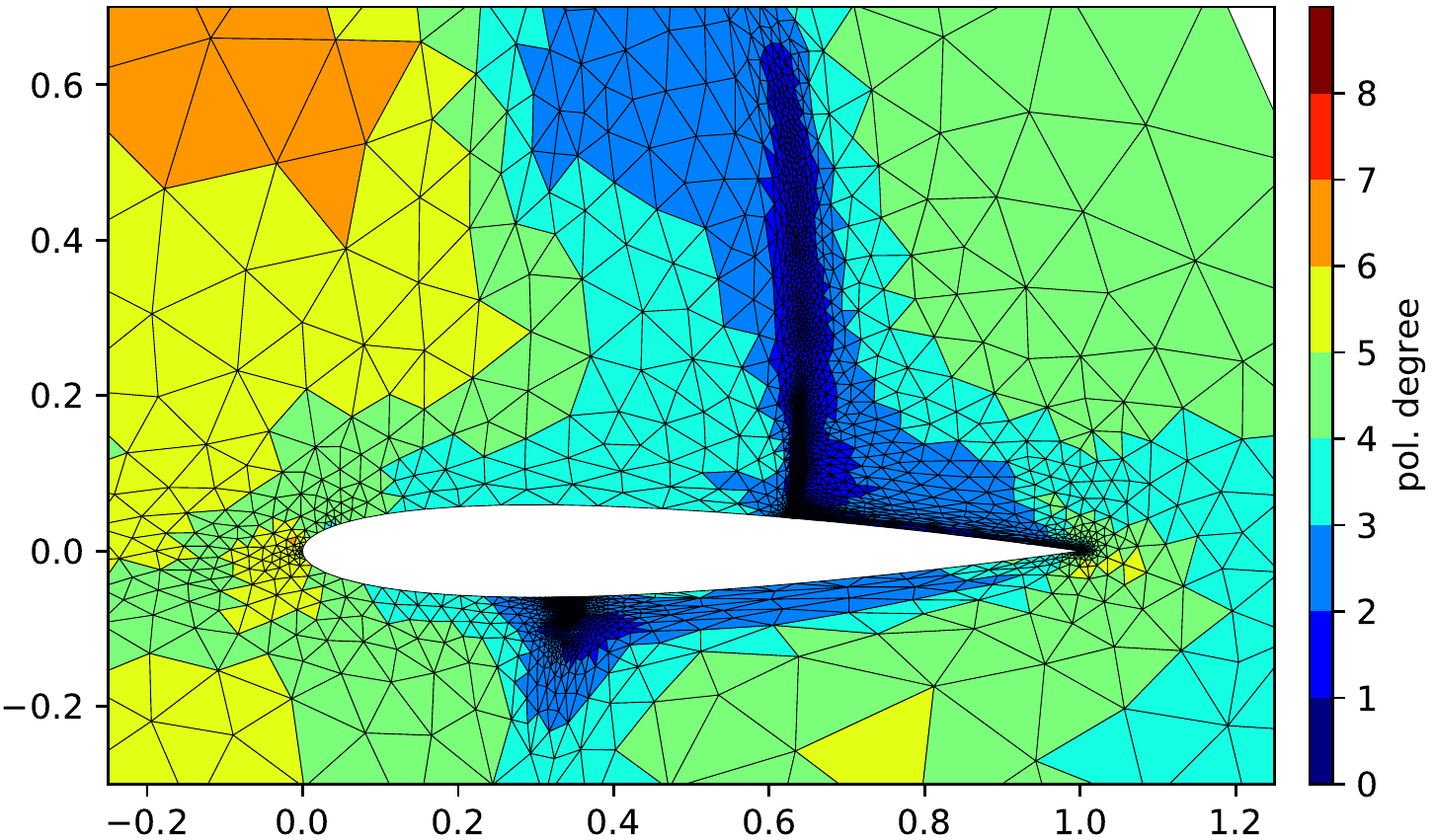}

  \vspace{2mm}

  \vertical{\hspace{7mm} Mach number}
  \includegraphics[width=0.46\textwidth]{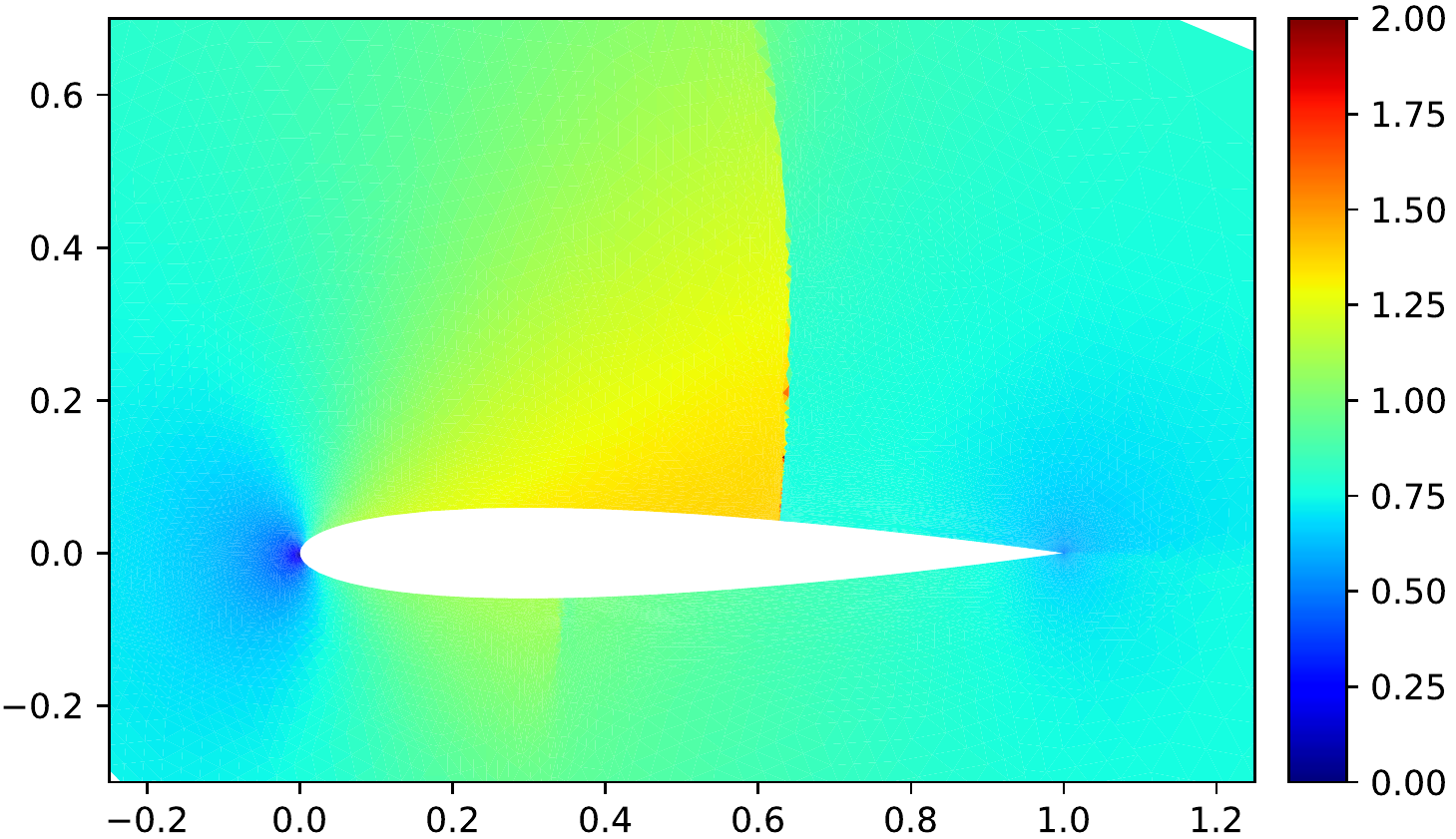}
  \hspace{1mm}
  \includegraphics[width=0.46\textwidth]{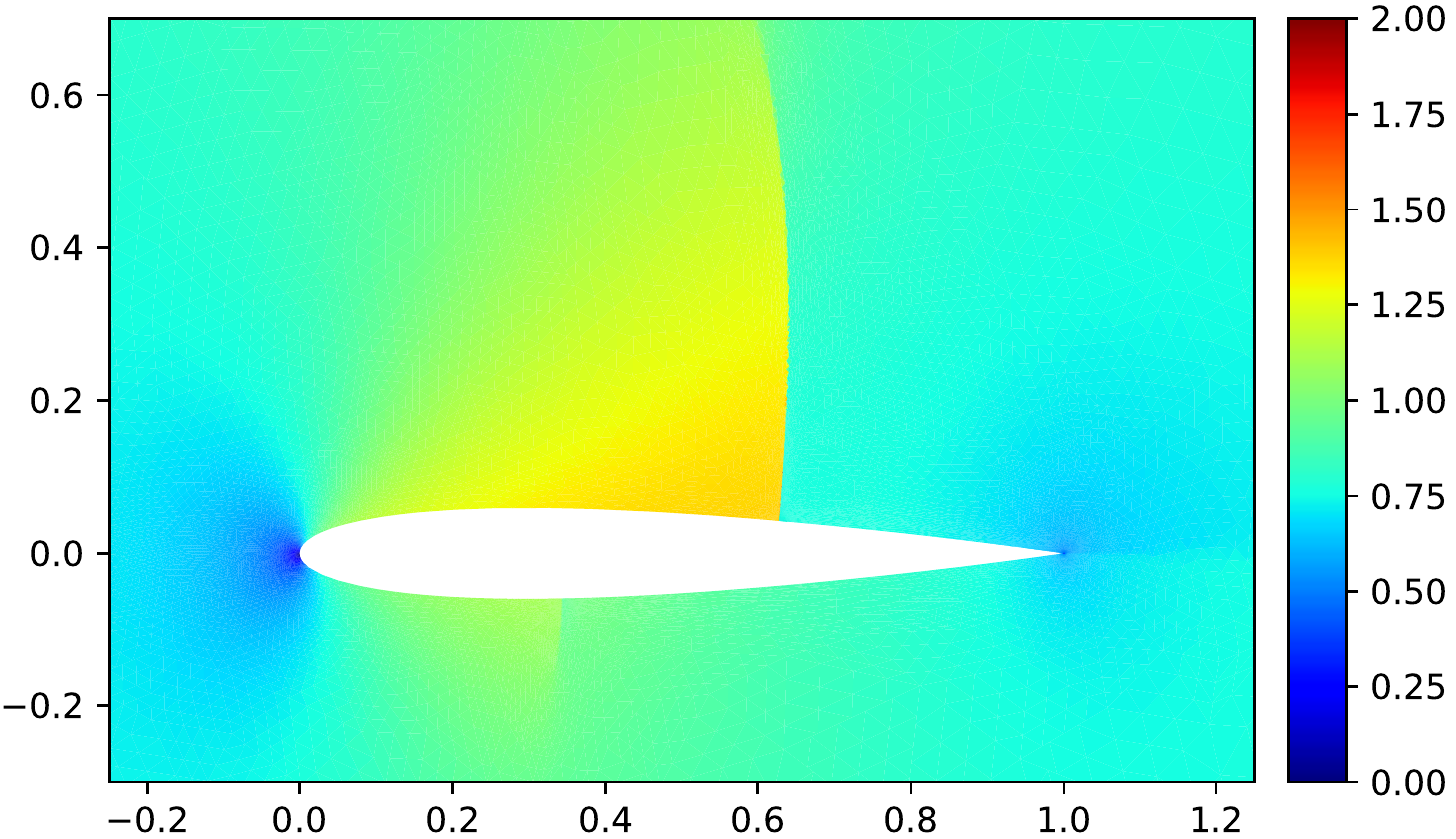}

  \vspace{2mm}

  \vertical{\hspace{1mm}first component of $\zzh$}
  \includegraphics[width=0.46\textwidth]{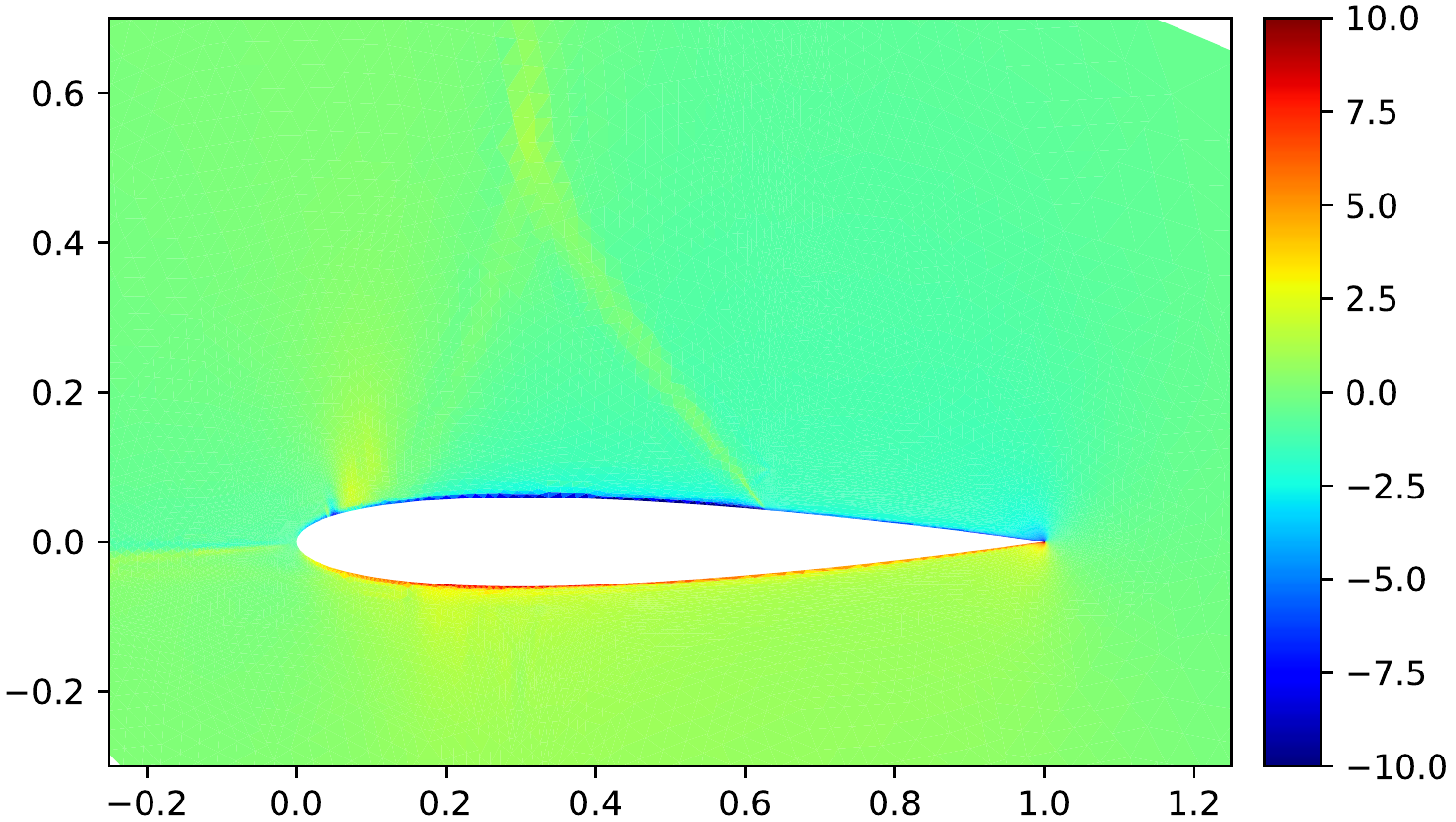}
  \hspace{1mm}
  \includegraphics[width=0.46\textwidth]{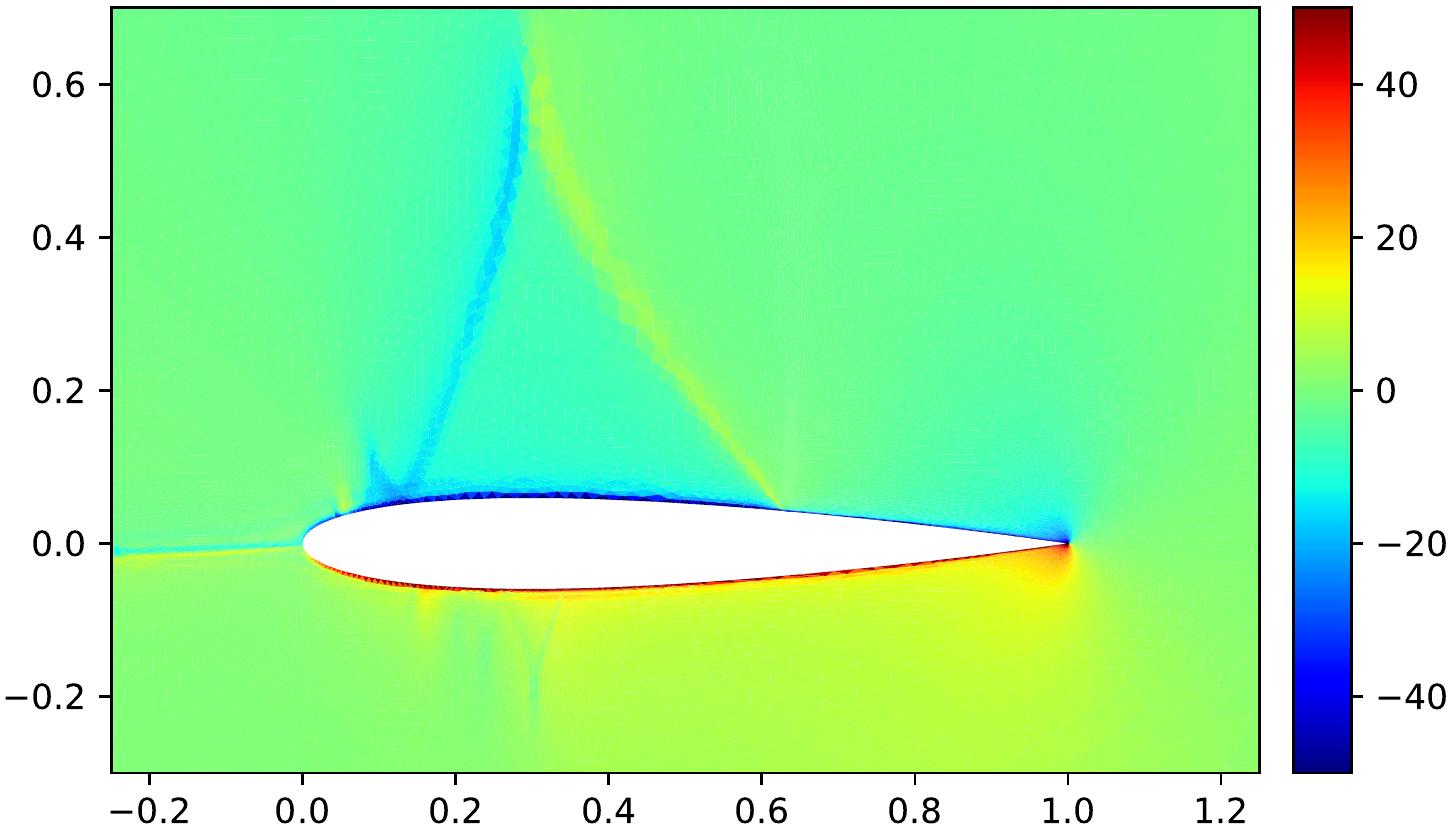}

   \caption{Transonic inviscid flow around the NACA 0012 profile
     ($\Min = 0.8$, $\alpha = 1.25\degree$): 
     $J$ equal to the 
     drag (left) and lift (right) coefficients using the $hp$-anisotropic refinement.}
   \label{fig:trans}
   \end{center}
\end{figure}

\section{Conclusion}
\label{sec:concl}
  
We presented the goal-oriented adaptive discontinuous Galerkin method for
the numerical solution of the Euler equations.
The DG discretization leads to the system of nonlinear algebraic equations
which are solved iteratively by an iterative solver based on
a suitable linearization of the numerical scheme.
This linearization is employed for the definition of the adjoint problem.
The careful treatment of the impermeable wall condition
and the  modification of target functional admit the
adjoint consistent discretization, which was proved analytically and supported by
numerical experiments. Therefore, iterative solvers not-based on the proper differentiation
can be used in the goal-oriented computations.

Furthermore, we extended 
the goal-oriented anisotropic $hp$-mesh adaptive technique from \cite{DWR_AMA,ESCO-18}
to the Euler equations. We presented numerical examples of subsonic as well as transonic
flows demonstrating the computational performance of this adaptive method.
Although we observed the exponential convergence of the error only for some numerical examples,
the potential of the $hp$-adaptive technique is obvious.


%
 \section*{Conflict of interest}

 The authors declare that they have no conflict of interest.


\end{document}